\documentclass{siamart220329}

\usepackage{lipsum}
\usepackage{amsfonts}
\usepackage{graphicx}
\usepackage{geometry}
\usepackage{epstopdf}
\usepackage{algorithmic}
\usepackage{amsopn}

\usepackage{bbm}
\usepackage{amsmath,amssymb,amscd}
\usepackage{graphicx}
\usepackage{indentfirst}
\usepackage{bm}
\usepackage[caption=false]{subfig}
\usepackage[utf8]{inputenc}

\graphicspath{{figure/}}

\newsiamremark{remark}{Remark}
\newsiamremark{hypothesis}{Hypothesis}
\crefname{hypothesis}{Hypothesis}{Hypotheses}
\newsiamthm{claim}{Claim}

\headers{Error analysis of pressure-correction scheme for NSPNP equations}{Y. He and H. Chen}

\title{Stability and error analysis of pressure-correction scheme for the Navier-Stokes-Planck-Nernst-Poisson equations 
	\thanks{This paper is supported by National Natural Science Foundation of China (No. 12371372), National Key R\&D Program of China (No. 2022YFA1004500), Natural Science Foundation of Fujian Province of China (No. 2021J01034) and Fundamental Research  Funds for the Central Universities (No. 20720220038).}}

\author{Yuyu He \thanks{School of Mathematical Sciences, Xiamen University, Xiamen, Fujian 361005, China; and Fujian Provincial Key Laboratory of Mathematical Modeling and High-Performance Scientific Computing, Xiamen University, Xiamen, Fujian 361005, China. E-mail address: \email{yyhe425@163.com (Y. He)}.} 
	\and Hongtao Chen \footnotemark[2] \thanks{Corresponding author. E-mail address: \email{chenht@xmu.edu.cn (H. Chen)}.}}

\usepackage{amsopn}

\begin{document}

\maketitle

\begin{abstract}
In this paper, we propose and analyze first-order time-stepping pressure-correction projection scheme for the Navier-Stokes-Planck-Nernst-Poisson equations. By introducing a governing equation for the auxiliary variable through the ionic concentration equations, we reconstruct the original equations into an equivalent system and develop a first-order decoupled and linearized scheme. This scheme preserves non-negativity and mass conservation of the concentration components and is unconditionally energy stable. We derive the rigorous error estimates in the two dimensional case for the ionic concentrations, electric potential, velocity and pressure in the $L^2$- and $H^1$-norms. Numerical examples are presented to validate the proposed scheme.
\end{abstract}



\section{Introduction}\label{se1}

We consider in this paper the numerical approximation of the following Navier-Stokes-Planck-Nernst-Poisson (NSPNP) equations
\begin{subequations}\label{hy1.1}
	\begin{align}
	\label{hy1.1a} &\partial_tc_1 + (\bm{u}\cdot\nabla) c_1 = \Delta c_1 + \nabla\cdot(c_1\nabla\phi),& \\
	\label{hy1.1b} &\partial_tc_2 + (\bm{u}\cdot\nabla) c_2 = \Delta c_2 - \nabla\cdot(c_2\nabla\phi),& \\
	\label{hy1.1c} &-\Delta\phi = c_1 - c_2,& \\
	\label{hy1.1d} &\partial_t\bm{u} + (\bm{u}\cdot\nabla)\bm{u} = \Delta \bm{u} - \nabla p - (c_1 - c_2)\nabla\phi,&\\
	\label{hy1.1e} &\nabla\cdot\bm{u} = 0,&
	\end{align}
\end{subequations}
where $c_1$ and $c_2$ represent the concentration functions of the positive and negative ions in the fluid respectively, $\phi$ is the electric potential function, $\bm{u}$ and $p$ denote the velocity field of the fluid and the pressure function respectively. We consider the non-slip boundary condition for $\bm{u}$, the homogeneous Neumann boundary condition for $c_1$, $c_2$ and $\phi$, i.e., all the fluxes vanish on the boundary $\partial\Omega$ of the bounded Lipschitz domain $\Omega\subset\mathbb{R}^d$ $(d=2,3)$
\begin{eqnarray}\label{hy1.2}
	\bm{u}|_{\partial\Omega} = 0, \quad \nabla c_1\cdot\bm{n}|_{\partial\Omega} = \nabla c_2\cdot\bm{n}|_{\partial\Omega} = \nabla\phi\cdot\bm{n}|_{\partial\Omega} = 0,
\end{eqnarray}
and impose the following zero mean condition for $\phi$ and $p$:
\begin{eqnarray}\label{hy1.3}
	\int_{\Omega}\phi(\bm{x},t)d\bm{x} = 0, \quad \int_{\Omega}p(\bm{x},t)d\bm{x} = 0.
\end{eqnarray}
For the non-negative regular initial ionic concentrations $c_1(\bm{x},0) \geq 0$ and $c_2(\bm{x},0) \geq 0$ almost everywhere in $\Omega$, the NSPNP system \eqref{hy1.1} admits solutions with non-negative ionic concentrations \cite{schmuck2009analysis}, i.e.,
\begin{eqnarray}\label{hy1.4}
	c_1(\bm{x},t) \geq 0, \quad c_2(\bm{x},t) \geq 0, \quad \forall\bm{x}\in\Omega,~ \forall t>0.
\end{eqnarray}
For either the boundary conditions \eqref{hy1.2}, one observes that the NSPNP system \eqref{hy1.1} is conserves the mass of each component, i.e.,
\begin{eqnarray}\label{hy1.5}
	\int_{\Omega}c_1(\bm{x},t)d\bm{x} = \int_{\Omega}c_1(\bm{x},0)d\bm{x}, \quad \int_{\Omega}c_2(\bm{x},t)d\bm{x} = \int_{\Omega}c_2(\bm{x},0)d\bm{x}, \quad \forall t>0.
\end{eqnarray}
In addition, the NSPNP system \eqref{hy1.1} also enjoys the energy dissipation law
\begin{eqnarray}\label{hy1.6}
	\frac{1}{2}\frac{d}{dt}\left(\|\bm{u}\|^2 + \|\nabla\phi\|^2\right) = - \|\nabla\bm{u}\|^2 - \|c_1-c_2\|^2 - \int_{\Omega}(c_1+c_2)|\nabla\phi|^2d\bm{x}, \quad \forall t>0,
\end{eqnarray}
where the $L^2$-inner product on $\Omega$ is denoted by $(\psi,\phi) = \int_{\Omega}\psi\phi d\bm{x}$ and the $L^2$-norm is denoted by $\|\phi\|=\sqrt{(\phi,\phi)}$.

The NSPNP system is frequently used to model the interactive behavior of charged colloidal particles, electro-hydro, and micro-/nano-fluidic dynamics. It combines three parts: (i) Planck-Nernst equations \eqref{hy1.1a}-\eqref{hy1.1b} for the ionic concentrations; (ii) Poisson equation \eqref{hy1.1c} for the internal electric potential; (iii) Navier-Stokes equations \eqref{hy1.1d}-\eqref{hy1.1e} modeling the movement of the fluid field under the action of the internal and external electric fields. In recent years, a large effort has been devoted to the mathematical analysis for the NSPNP system \eqref{hy1.1}, see, for instance, \cite{deng2011well,schmuck2009analysis,shen2022stability,wang2019quasi,wang2016generalized} and the references therein. Numerically, the main challenge in the numerical approximation for solving the NSPNP system is to construct efficient scheme that can verify the properties of the non-negativity and mass conservation of the concentration components and the energy stability at the discrete level, as it involves nonlinear coupling between the ionic concentration, electric potential and velocity.

For the PNP equations, several structure-preserving numerical methods have been developed recently, for example, \cite{dong2024positivity,fu2022high,gao2017linearized,he2019positivity,hu2020fully,liu2021positivity,liu2023second,shen2021unconditionally} and the references therein. There is also considerable research on numerical studies for the NSPNP equations. Proh and Schmuck \cite{prohl2010convergent} were proposed and analyzed fully discrete finite element scheme, which satisfies the energy and entropy dissipation laws. Bauer et al. \cite{bauer2012stabilized} were introduced the stability term to prevent potential pseudo-oscillations in convection-dominated cases using standard Galerkin finite element methods. Metti et al. \cite{metti2016energetically} were proposed the finite element discretization of lines approached combining with the discontinuous Galerkin method. Liu and Xu \cite{liu2017efficient} were constructed and analyzed stable first-order time-stepping schemes for the time discretization and used a spectral method for the spatial discretization. Dehghan et al. \cite{dehghan2023optimal} were proposed and analyzed a fully coupled, nonlinear, and energy stable virtual element method. Yu et al. \cite{yu2023positivity} were constructed fully discrete scheme that enjoys the properties of positivity preserving, mass conserving, and unconditionally energy stability. In \cite{he2021mixed,li2024error}, some stabilized mixed finite element methods and its error analysis for the NSPNP equations were proposed. Until recently, by introducing scalar auxiliary variable (SAV) approaches and utilizing ``zero-energy-contribution'' property, some linearized, second-order accurate, positivity-preserving and unconditionally energy stable time-stepping schemes for the NSPNP equations were presented in \cite{pan2024linear,zhou2023efficient}. However, there seems to be no properties-preserving (including non-negativity and mass conservation for ionic concentrations), fully-decoupled, linearized and unconditionally energy stable scheme and its error analysis in the literature.

In this paper, the main purposes are to construct a fully-decoupled, linearized, mass- and non-negative preserving, and unconditionally energy-stable scheme for the NSPNP system \eqref{hy1.1} and to carry out a rigorous error analysis for the proposed scheme. Our main contributions include: 
\begin{itemize}
	\item We construct new first-order SAV pressure-correction scheme. Different from the SAV schemes in \cite{pan2024linear,zhou2023efficient}, our idea is to use a suitable SAV, deriving from the ionic concentration equations, to treat the nonlinear terms in the Navier-Stokes equations part. We describe the solution procedure and derive the mass conservation and non-negative properties and the unconditional energy stability. 
	\item We carry out a rigorous error analysis for the proposed scheme in the two-dimensional case. The error analysis uses essentially the $L^{\infty}$ bounds for $c_1,c_2$ and $H^1$ bound for $\bm{u}$, which are not available through energy stability. Therefore, we adopt mathematical induction method to derive the error estimates of the ionic concentrations, electric potential, velocity and pressure in the $L^2$- and $H^1$-norms. 
\end{itemize} 

The structure of this paper is arranged as follows. In Section \ref{se2}, we construct a first-order SAV pressure-correction time-stepping scheme and derive the properties of mass and non-negativity and the unconditional energy stability. In Section \ref{se3}, we carry out a rigorous error analysis by using mathematical induction for the proposed scheme in the two-dimensional case. In Section \ref{se4}, we present fully-discrete pressure-correction projection scheme and provide some numerical experiments to validate the proposed scheme. Finally, the conclusion is given in Section \ref{se5}.

\section{The SAV pressure-correction scheme and its stability}\label{se2}

In this section, we construct first-order explicit-implicit pressure-correction scheme based on the SAV approach for the NSPNP system \eqref{hy1.1}, and show that the proposed scheme is positivity-preserving, mass conservative for ionic concentrations and unconditionally energy stable.

Taking the inner product of \eqref{hy1.1a} and \eqref{hy1.1b} with $\phi$ and of \eqref{hy1.1c} with $c_1-c_2$, and together with them, we obtain
\begin{eqnarray}\label{hy2.1}
	(\partial_t(c_1-c_2),\phi) = ((c_1-c_2)\bm{u},\nabla\phi) - \|c_1-c_2\|^2 - \int_{\Omega}(c_1+c_2)|\nabla\phi|^2d\bm{x}.
\end{eqnarray}
Furthermore, we use \eqref{hy1.1c} to get
\begin{eqnarray}\label{hy2.2}
	(\partial_t(c_1-c_2),\phi) = - (\partial_t\Delta\phi,\phi) = \frac{1}{2}\frac{d}{dt}\|\nabla\phi\|^2.
\end{eqnarray}
Combining \eqref{hy2.1} with \eqref{hy2.2}, we have
\begin{eqnarray}\label{hy2.3}
	\frac{1}{2}\frac{d}{dt}\|\nabla\phi\|^2 = ((c_1-c_2)\bm{u},\nabla\phi) - \|c_1-c_2\|^2 - \int_{\Omega}(c_1+c_2)|\nabla\phi|^2d\bm{x}.
\end{eqnarray}

We notice that there exists a positive constant $C_0>0$ such that the partial energy function $E(\phi):=\frac{1}{2}\|\nabla\phi\|^2 + C_0 \geq 1$. By introducing a time-dependent scalar auxiliary function $r(t) := \sqrt{E(\phi)}$, we have 
\begin{eqnarray}\label{hy2.4}
	2r\frac{dr}{dt} = \frac{d}{dt}E(\phi).
\end{eqnarray}
By adding the zero-value term $(\bm{u}\cdot\nabla\bm{u},\bm{u})$ into \eqref{hy2.4} and multiplying the factor $r/\sqrt{E(\phi)}$, the governing equation for the auxiliary variable $r(t)$ can be obtained by \eqref{hy2.3} as follows:
\begin{eqnarray}\label{hy2.5}
	2r\frac{dr}{dt} = \frac{r}{\sqrt{E(\phi)}}\big((\bm{u}\cdot\nabla\bm{u},\bm{u}) + ((c_1-c_2)\nabla\phi,\bm{u})\big) - \|c_1-c_2\|^2 - \int_{\Omega}(c_1+c_2)|\nabla\phi|^2d\bm{x}.
\end{eqnarray}
Then we rewrite the NSPNP system \eqref{hy1.1} into the following equivalent system
\begin{subequations}\label{hy2.6}
	\begin{align}
		\label{hy2.6a} &\partial_tc_1 - \Delta c_1 + \bm{u}\cdot\nabla c_1 - \nabla\cdot(c_1\nabla\phi) = 0,& \\
		\label{hy2.6b} &\partial_tc_2 - \Delta c_2 + \bm{u}\cdot\nabla c_2 + \nabla\cdot(c_2\nabla\phi) = 0,& \\
		\label{hy2.6c} &-\Delta\phi - c_1 + c_2 = 0,& \\
		\label{hy2.6d} &\partial_t\bm{u} - \Delta \bm{u} + \nabla p +  \frac{r}{\sqrt{E(\phi)}}\big(\bm{u}\cdot\nabla\bm{u} + (c_1 - c_2)\nabla\phi\big) = 0,&\\
		\label{hy2.6e} &\nabla\cdot\bm{u} = 0,& \\
		\label{hy2.6f} &\frac{dr}{dt} = \frac{1}{2\sqrt{E(\phi)}}\big((\bm{u}\cdot\nabla\bm{u},\bm{u}) +  ((c_1-c_2)\nabla\phi,\bm{u})\big) & \nonumber \\
		& \qquad - \frac{1}{2r}\Big(\|c_1-c_2\|^2 + \int_{\Omega}(c_1+c_2)|\nabla\phi|^2d\bm{x}\Big).&
	\end{align}
\end{subequations}
Noticing that the term $(\bm{u}\cdot\nabla\bm{u},\bm{u}) \equiv 0$ and the factor $r/\sqrt{E(\phi)} \equiv 1$, it is clear that the above system is strictly equivalent to the original system. In the discretized case, the first and second terms of the right hand side of \eqref{hy2.6f} can balance the nonlinear terms in \eqref{hy2.6d}.  In addition, we have a modified energy dissipation law
\begin{eqnarray}\label{hy2.7}
	\frac{d}{dt}\left(\frac{1}{2}\|\bm{u}\|^2 + r^2\right) = - \|\nabla\bm{u}\|^2 - \|c_1-c_2\|^2 - \int_{\Omega}(c_1+c_2)|\nabla\phi|^2d\bm{x}, \quad \forall t>0.
\end{eqnarray}

Let $\tau$ be the time-step size and $t^n = n\tau$ for $0 \leq n \leq N$, where $N$ is integer. We define the difference operators for $n \leq N-1$:
$$
\delta_tf^{n+1} = \frac{f^{n+1}-f^n}{\tau}, \quad \delta^2_tf^{n+1} = \frac{\delta_tf^{n+1} - \delta_tf^n}{\tau} = \frac{f^{n+1} -2f^n - f^{n-1}}{2\tau^2}.
$$
We construct the following first-order explicit-implicit pressure-correction scheme: 

Give initial values $c^0_1$, $c^0_2$, $\bm{u}^0$ and $r^0$, find $\phi^0$ vy solving 
\begin{eqnarray}
	-\Delta\phi^0 = c^0_1 - c^0_2, \quad \nabla\phi^0\cdot\bm{n}|_{\partial\Omega} = 0.
\end{eqnarray}
Then give an approximate initial $p^0$, find $(c^{n+1}_1,c^{n+1}_2,\phi^{n+1},\bar{\bm{u}}^{n+1},\bm{u}^{n+1},p^{n+1},r^{n+1})$ by solving
\begin{subequations}\label{hy2.8}
	\begin{align}
	\label{hy2.8a} &\delta_tc^{n+1}_1 - \Delta c^{n+1}_1 + \bm{u}^n\cdot\nabla c^{n+1}_1 - \nabla\cdot(c^{n+1}_1\nabla\phi^n) =  0,& \\
	\label{hy2.8b} &\delta_tc^{n+1}_2 - \Delta c^{n+1}_2 + \bm{u}^n\cdot\nabla c^{n+1}_2 + \nabla\cdot(c^{n+1}_2\nabla\phi^n) =  0,& \\
	\label{hy2.8c} &-\Delta\phi^{n+1} = c^{n+1}_1 - c^{n+1}_2,& \\
	\label{hy2.8d} &\frac{\bar{\bm{u}}^{n+1}-\bm{u}^n}{\tau} - \Delta\bar{\bm{u}}^{n+1} + \nabla p^n + \frac{r^{n+1}}{\sqrt{E(\phi^{n+1})}} \Big(\bm{u}^n\cdot\nabla\bm{u}^n + (c^{n}_1 - c^{n}_2)\nabla\phi^{n}\Big) = 0,& \\
	\label{hy2.8e} &\frac{\bm{u}^{n+1} - \bar{\bm{u}}^{n+1}}{\tau} + \nabla p^{n+1} - \nabla p^n = 0,& \\
	\label{hy2.8f} &\nabla\cdot\bm{u}^{n+1} = 0,& \\
	\label{hy2.8g} &\delta_tr^{n+1} = \frac{1}{2\sqrt{E(\phi^{n+1})}} \Big((\bm{u}^n\cdot\nabla\bm{u}^n,\bar{\bm{u}}^{n+1}) + (c^{n}_1-c^{n}_2)\nabla\phi^{n},\bar{\bm{u}}^{n+1})\Big) & \nonumber \\
	& \qquad \qquad - \frac{1}{2r^{n+1}}\Big(\|c^{n+1}_1-c^{n+1}_2\|^2 + \int_{\Omega}(c^{n+1}_1+c^{n+1}_2)|\nabla\phi^{n+1}|^2d\bm{x}\Big),&
	\end{align}
\end{subequations}
with the discrete boundary conditions
\begin{eqnarray}\label{hy2.9}
	\bm{u}^{n+1}|_{\partial\Omega} = \bar{\bm{u}}^{n+1}|_{\partial\Omega} = 0, \quad 
	\nabla c^{n+1}_1\cdot\bm{n}|_{\partial\Omega} = \nabla c^{n+1}_2\cdot\bm{n}|_{\partial\Omega} = \nabla\phi^{n+1}\cdot\bm{n}|_{\partial\Omega} = 0.
\end{eqnarray}
One can observe that the solution $(c^{n+1}_1,c^{n+1}_2,\phi^{n+1},\bm{u}^{n+1},p^{n+1})$ of \eqref{hy2.8a}-\eqref{hy2.8c}, \eqref{hy2.8e} and \eqref{hy2.8f} is fully decoupled, but the solution $(\bar{\bm{u}}^{n+1},r^{n+1})$ of \eqref{hy2.8d} and \eqref{hy2.8g} needs to be solved decoupling by the following splitting method.
 
We denote $\xi^{n+1} = r^{n+1}/\sqrt{E(\phi^{n+1})}$ and set
\begin{eqnarray}\label{hy2.10}
	\bar{\bm{u}}^{n+1} = \bar{\bm{u}}^{n+1}_1 + \xi^{n+1}\bar{\bm{u}}^{n+1}_2,
\end{eqnarray}
in \eqref{hy2.8d}. Then we can determine $\bar{\bm{u}}^{n+1}_i$ $(i=1,2)$ from
\begin{eqnarray}
	\label{hy2.11} && \frac{\bar{\bm{u}}^{n+1}_1 - \bm{u}^n}{\tau} - \Delta\bar{\bm{u}}^{n+1}_1 + \nabla p^n = 0, \quad \text{with} \quad \bar{\bm{u}}^{n+1}_1|_{\partial\Omega} = 0, \\
	\label{hy2.12} && \frac{\bar{\bm{u}}^{n+1}_2}{\tau} - \Delta\bar{\bm{u}}^{n+1}_2 + \bm{u}^n\cdot\nabla\bm{u}^n + (c^{n}_1 - c^{n}_2)\nabla\phi^{n} = 0, \quad \text{with} \quad \bar{\bm{u}}^{n+1}_2|_{\partial\Omega} = 0.
\end{eqnarray}
Once $\bar{\bm{u}}^{n+1}_i$ $(i=1,2)$ are known, we can plug \eqref{hy2.10} and $r^{n+1} = \xi^{n+1}\sqrt{E(\phi^{n+1})}$ in \eqref{hy2.8g} to determine $\xi^{n+1}$ from 
\begin{eqnarray}\label{hy2.13}
	a^n(\xi^{n+1})^2 - b^n\xi^{n+1} + c^n = 0,
\end{eqnarray}
where
\begin{eqnarray*}
	&& a^n = 2E(\phi^{n+1}) - \tau(\bm{u}^n\cdot\nabla\bm{u}^n,\bar{\bm{u}}^{n+1}_2) - \tau((c^{n}_1-c^{n}_2)\nabla\phi^{n},\bar{\bm{u}}^{n+1}_2), \\
	&& b^n = 2r^n\sqrt{E(\phi^{n+1})} + \tau(\bm{u}^n\cdot\nabla\bm{u}^n, \bar{\bm{u}}^{n+1}_1) + \tau((c^{n}_1-c^{n}_2)\nabla\phi^{n},\bar{\bm{u}}^{n+1}_1), \\
	&& c^n = \tau\|c^{n+1}_1-c^{n+1}_2\|^2 + \tau\int_{\Omega}(c^{n+1}_1+c^{n+1}_2)|\nabla\phi^{n+1}|^2d\bm{x}.
\end{eqnarray*}
We observe that \eqref{hy2.13} is a quadratic equation for $\xi^{n+1}$, which can be solved directly by using the quadratic formula. The nonlinear quadratic equation \eqref{hy2.13} has two solutions. Due to the exact solution is 1, we should choose the root which is closer to 1. Once $\xi^{n+1}$ is known, we can determine $r^{n+1} = \xi^{n+1}\sqrt{E(\phi^{n+1})}$ and $\bar{\bm{u}}^{n+1}$ by \eqref{hy2.10}.

We show below that the proposed first-order SAV pressure-correction scheme \eqref{hy2.8} is unconditionally energy stable.

\begin{theorem}\label{th2.1}
Give the initial values $c^0_1(\bm{x})$, $c^0_2(\bm{x})$ which are non-negative almost everywhere in $\Omega$, then the scheme \eqref{hy2.8} has the following properties for all $0 \leq n \leq N-1$:
\begin{itemize}
	\item Non-negativity:
	\begin{eqnarray}\label{hy2.14}
		c^{n+1}_1(\bm{x}) \geq 0, \quad c^{n+1}_2(\bm{x}) \geq 0, \quad \forall\bm{x}\in\Omega.
	\end{eqnarray}
	\item Mass conservation:
	\begin{eqnarray}\label{hy2.15}
		\int_{\Omega}c^{n+1}_1(\bm{x})d\bm{x} = \int_{\Omega}c^{n}_1(\bm{x})d\bm{x}, \quad \int_{\Omega}c^{n+1}_1(\bm{x})d\bm{x} = \int_{\Omega}c^{n}_1(\bm{x})d\bm{x}.
	\end{eqnarray}
	\item Unconditional energy stability:
	\begin{eqnarray}\label{hy2.16}
		E^{n+1} - E^n \leq -\tau\|\nabla\bar{\bm{u}}^{n+1}\|^2 - \tau\|c^{n+1}_1-c^{n+1}_2\|^2 - \tau\|\sqrt{c^{n+1}_1+c^{n+1}_2}\nabla\phi^{n+1}\|^2,
	\end{eqnarray}
	where 
	$$
	E^{n+1} = \frac{1}{2}\|\bm{u}^{n+1}\|^2 + \frac{\tau^2}{2}\|\nabla p^{n+1}\|^2 + |r^{n+1}|^2.
	$$
\end{itemize}
\end{theorem}

\begin{proof}
The non-negativity \eqref{hy2.14} of the discrete ionic concentrations $c^{n+1}_1(\bm{x})$, $c^{n+1}_2(\bm{x})$ can be proved by exactly the same lines as Lemma 3.1 in \cite{liu2017efficient}. Integrating \eqref{hy2.8a} and \eqref{hy2.8b} over $\Omega$ and using the condition \eqref{hy2.9}, we immediately obtain the mass conservation laws \eqref{hy2.15}.

Taking the inner product of \eqref{hy2.8d} with $\bar{\bm{u}}^{n+1}$, we have
\begin{eqnarray}\label{hy2.17}
	\begin{aligned}
	& \frac{1}{2\tau}\left(\|\bar{\bm{u}}^{n+1}\|^2 - \|\bm{u}^n\|^2 + \|\bar{\bm{u}}^{n+1}-\bm{u}^n\|^2\right) + \|\nabla\bar{\bm{u}}^{n+1}\|^2 + (\nabla p^n,\bar{\bm{u}}^{n+1}) \\
	=& - \frac{r^{n+1}}{\sqrt{E(\phi^{n+1})}}\Big((\bm{u}^n\cdot\nabla\bm{u}^n,\bar{\bm{u}}^{n+1}) + ((c^{n}_1-c^{n}_2)\nabla\phi^{n},\bar{\bm{u}}^{n+1})\Big).
	\end{aligned}
\end{eqnarray}
Recalling \eqref{hy2.8e}, we have
\begin{eqnarray}\label{hy2.18}
	\bm{u}^{n+1} + \tau\nabla p^{n+1} = \bar{\bm{u}}^{n+1} + \tau\nabla p^n.
\end{eqnarray}
Taking the inner product of \eqref{hy2.18} with itself on both sides, we have
\begin{eqnarray}\label{hy2.19}
	\|\bm{u}^{n+1}\|^2 + \tau^2\|\nabla p^{n+1}\|^2 = \|\bar{\bm{u}}^{n+1}\|^2 + \tau^2\|\nabla p^n\|^2 + 2\tau(\nabla p^n,\bar{\bm{u}}^{n+1}),
\end{eqnarray}
where we use \eqref{hy2.8f} and the term $(\nabla p^{n+1},\bm{u}^{n+1}) = -(p^{n+1},\nabla\cdot\bm{u}^{n+1}) = 0$. Combining \eqref{hy2.17} with \eqref{hy2.19}, we obtain 
\begin{eqnarray}\label{hy2.20}
	\begin{aligned}
	& \frac{1}{2\tau}\left(\|\bm{u}^{n+1}\|^2 - \|\bm{u}^n\|^2 + \|\bar{\bm{u}}^{n+1}-\bm{u}^n\|^2\right) + \frac{\tau}{2}\left(\|\nabla p^{n+1}\|^2 - \|\nabla p^n\|^2\right) + \|\nabla\bar{\bm{u}}^{n+1}\|^2 \\
	=& - \frac{r^{n+1}}{\sqrt{E(\phi^{n+1})}} \Big((\bm{u}^n\cdot\nabla\bm{u}^n,\bar{\bm{u}}^{n+1}) + ((c^{n}_1-c^{n}_2)\nabla\phi^{n},\bar{\bm{u}}^{n+1})\Big).
    \end{aligned}
\end{eqnarray}
Multiplying \eqref{hy2.8g} with $2r^{n+1}$, we have
\begin{eqnarray}\label{hy2.21}
	\begin{aligned}
	\frac{1}{\tau}\left(|r^{n+1}|^2 - |r^n|^2 + |r^{n+1}-r^n|^2\right) = - \|c^{n+1}_1-c^{n+1}_2\|^2 - \|\sqrt{c^{n+1}_1+c^{n+1}_2}\nabla\phi^{n+1}\|^2  \\
	+ \frac{r^{n+1}}{\sqrt{E(\phi^{n+1})}} \Big((\bm{u}^n\cdot\nabla\bm{u}^n,\bar{\bm{u}}^{n+1}) + ((c^{n}_1-c^{n}_2)\nabla\phi^{n},\bar{\bm{u}}^{n+1})\Big),
\end{aligned}
\end{eqnarray}
Together \eqref{hy2.20} with \eqref{hy2.21}, we obtain
\begin{eqnarray*}
	&& \frac{1}{2}\left(\|\bm{u}^{n+1}\|^2 - \|\bm{u}^n\|^2 + \|\bar{\bm{u}}^{n+1}-\bm{u}^n\|^2\right) + \frac{\tau^2}{2}\left(\|\nabla p^{n+1}\|^2 - \|\nabla p^n\|^2\right) + |r^{n+1}|^2 - |r^n|^2  \\
	&& \quad  + |r^{n+1}-r^n|^2 + \tau\|\nabla\bar{\bm{u}}^{n+1}\|^2 + \tau\|c^{n+1}_1-c^{n+1}_2\|^2 + \tau\|\sqrt{c^{n+1}_1+c^{n+1}_2}\nabla\phi^{n+1}\|^2 = 0,
\end{eqnarray*}
which implies the desired result \eqref{hy2.16}.
\end{proof}

\begin{remark}
	Note that in the case of inhomogeneous Planck-Nernst equations
	\begin{eqnarray*}
		&& \partial_tc_1 - \Delta c_1 + \bm{u}\cdot\nabla c_1 - \nabla\cdot(c_1\nabla\phi) = f_{c_1}, \\
		&& \partial_tc_2 - \Delta c_2 + \bm{u}\cdot\nabla c_2 + \nabla\cdot(c_2\nabla\phi) = f_{c_2},
	\end{eqnarray*}
   we only need to slightly replace \eqref{hy2.6f} to
   \begin{eqnarray*}
   	\frac{dr}{dt} &=& \frac{1}{2\sqrt{E(\phi)}}\big((\bm{u}\cdot\nabla\bm{u},\bm{u}) +  ((c_1-c_2)\nabla\phi,\bm{u})\big) \\
   	&& - \frac{1}{2r}\Big(\|c_1-c_2\|^2 + \int_{\Omega}(c_1+c_2)|\nabla\phi|^2d\bm{x} - \int_{\Omega}(f_{c_1} - f_{c_2})\phi d\bm{x}\Big).
   \end{eqnarray*}
   In \eqref{hy2.13}, we need to modify $c^n$ to 
   $$
   c^n = \tau\|c^{n+1}_1-c^{n+1}_2\|^2 + \tau\int_{\Omega}(c^{n+1}_1+c^{n+1}_2)|\nabla\phi^{n+1}|^2d\bm{x} - \tau\int_{\Omega}(f^{n+1}_{c_1} - f^{n+1}_{c_2})\phi^{n+1} d\bm{x}.
   $$
\end{remark}

\begin{remark}
	We adopt the same temporal discretization method as in \cite{liu2017efficient} for \eqref{hy2.6a}-\eqref{hy2.6b}. However, the scheme in \cite{liu2017efficient} is stable under condition $\tau \leq \|c^{n+1}_1 + c^{n+1}_2\|^{-1}_{L^{\infty}}$ due to the nonlinear terms, which brings a great obstacle to error analysis. In fact, our scheme \eqref{hy2.8} essentially removes this stability condition by exploiting the property that the nonlinear contributions of the energy can cancel each other out.
\end{remark}

\section{Error analysis}\label{se3}

In this section, we carry out a rigorous error analysis for the proposed first-order scheme \eqref{hy2.8} in the two dimensional case. Throughout this section, we denote by $C$ a general constant independent of time step $\tau$ but may have different values at different occurrences.

Let $\Omega$ be a bounded domain with Lipschitz boundary in $\mathbb{R}^d$ ($d=2$). For the integer $k\geq 1$ and $1 \leq p \leq +\infty$, we denote by $W^{k,p}(\Omega)$ and $W^{k,p}_0(\Omega)$ the usual Sobolev spaces and set $\bm{W}^{k,p}(\Omega)=(W^{k,p}(\Omega))^d$ and $\bm{W}^{k,p}_0(\Omega)=(W^{k,p}_0(\Omega))^d$. Especially, $W^{0,p}(\Omega)$ is the Lebesgue space $L^p(\Omega)$ when $k=0$ and $W^{k,2}(\Omega)$ is the Hillbert space $H^k(\Omega)$ when $p=2$. Other usual spaces are defined as
\begin{eqnarray*}
	&\bm{H}^1_0(\Omega)=\left\{\bm{v}\in\bm{H}^1(\Omega),\ \bm{v}|_{\partial\Omega}=0\right\},
	\quad L^2_0(\Omega)=\Big\{p\in L^2(\Omega),\ \int_{\Omega}pdx=0\Big\},&\\
	&\bm{V}(\Omega)=\left\{\bm{v}\in \bm{H}^1_0(\Omega),\ \nabla\cdot\bm{v}=0\right\},\quad
	\bm{H}(\Omega)=\left\{\bm{u}\in\bm{L}^2(\Omega)\ |\ \nabla\cdot\bm{u}=0,
	\ \bm{u}\cdot\bm{n}|_{\partial\Omega}=0\right\}.&
\end{eqnarray*}
The norms corresponding to $W^{k,p}(\Omega)$, $H^k(\Omega)$ and $L^p(\Omega)$ will be denoted simply by $\|\cdot\|_{W^{k,p}}$, $\|\cdot\|_{H^k}$ and $\|\cdot\|_{L^p}$, respectively. In particular, the $L^2$-norm $\|\cdot\|_{L^2}$ is simply rewritten as $\|\cdot\|$.

We define the Stokes operator \cite{temam2001navier}
$$
\bm{A}\bm{u} = -\mathbb{P}_H\Delta\bm{u}, \quad \forall\bm{u} \in\bm{D}(\bm{A}) := \bm{H}^2(\Omega)\cap\bm{V}(\Omega),
$$
where $\mathbb{P}_H$ is the orthogonal project in $\bm{L}^2(\Omega)$ onto $\bm{H}(\Omega)$. Obviously $\mathbb{P}_H$ is continuous into $\bm{L}^2(\Omega)$. In fact $\mathbb{P}_H$ also maps $\bm{H}^1(\Omega)$ into itself and is continuous for the norm of $\bm{H}^1(\Omega)$. We derive from the above that 
\begin{eqnarray}\label{hy3.1} 
	\|\nabla\bm{u}\| \leq C_1\|\bm{A}^{\frac{1}{2}}\bm{u}\|, \quad \|\nabla\bm{u}\| \leq C_1\|\bm{A}\bm{u}\|, \quad \forall\bm{u}\in\bm{D}(\bm{A}),
\end{eqnarray}
and give the following well-known inequalities with $d=2$ which will be used in the sequel
\begin{eqnarray}
	\label{hy3.2} && \|\bm{u}\| \leq C_1\|\nabla\bm{u}\|, \quad \forall\bm{u}\in\bm{H}^1_0(\Omega),\\
	\label{hy3.3} &&\|\bm{u}\|_{L^4} \leq C_1\|\bm{u}\|^{\frac{1}{2}}\|\nabla\bm{u}\|^{\frac{1}{2}}, \quad \forall\bm{u}\in\bm{H}^1(\Omega), \\
	\label{hy3.4} &&\|\bm{u}\|_{L^{\infty}} \leq C_1\|\nabla\bm{u}\|^{\frac{1}{2}}\|\Delta\bm{u}\|^{\frac{1}{2}}, \quad \forall\bm{u}\in\bm{H}^2(\Omega),
\end{eqnarray}
where $C_1$ is a positive constant depending only on $\Omega$.

We define the trilinear form $\bm{b}(\cdot,\cdot,\cdot)$ by
$$
\bm{b}(\bm{u},\bm{v},\bm{w}) = \int_{\Omega}(\bm{u}\cdot\nabla)\bm{v}\cdot\bm{w}d\bm{x},
$$
then by  using a combination of integration by parts, Holder's inequality, and Sobolev inequalities \cite{temam1995navier} with $d=2$, we have
\begin{eqnarray}\label{hy3.5}
	\bm{b}(\bm{u},\bm{v},\bm{w})\leq
	\left\{
	\begin{array}{l}
		C_2\|\bm{u}\|_{H^1}\|\bm{v}\|_{H^1}\|\bm{w}\|_{H^1},\ \forall\bm{u},\bm{v}\in\bm{H}(\Omega),
		\bm{w}\in \bm{H}^1_0(\Omega),\\
		C_2\|\bm{u}\|_{H^2}\|\bm{v}\|\|\bm{w}\|_{H^1},\ \forall\bm{u}\in \bm{H}^2(\Omega)\cap\bm{H}(\Omega),
		\bm{v}\in\bm{H}(\Omega), \bm{w}\in\bm{H}^1_0(\Omega),\\
		C_2\|\bm{u}\|_{H^2}\|\bm{v}\|_{H^1}\|\bm{w}\|,\ \forall\bm{u}\in \bm{H}^2(\Omega)\cap\bm{H}(\Omega),
		\bm{v}\in\bm{H}(\Omega), \bm{w}\in\bm{H}^1_0(\Omega),\\
		C_2\|\bm{u}\|_{H^1}\|\bm{v}\|_{H^2}\|\bm{w}\|,\ \forall\bm{v}\in\bm{H}^2(\Omega)\cap\bm{H}(\Omega),
		\bm{u}\in\bm{H}(\Omega), \bm{w}\in\bm{H}^1_0(\Omega),\\
		C_2\|\bm{u}\|\|\bm{v}\|_{H^2}\|\bm{w}\|_{H^1},\ \forall\bm{v}\in\bm{H}^2(\Omega)\cap\bm{H}(\Omega),
		\bm{u}\in\bm{H}(\Omega), \bm{w}\in\bm{H}^1_0(\Omega),\\
		C_2\|\bm{u}\|^{\frac{1}{2}}_{H^1}\|\bm{u}\|^{\frac{1}{2}}\|\bm{v}\|^{\frac{1}{2}}_{H^1}
		\|\bm{v}\|^{\frac{1}{2}}\|\bm{w}\|_{H^1},\ \forall\bm{u},\bm{v}\in\bm{H}(\Omega),
		\bm{w}\in\bm{H}^1_0(\Omega),
	\end{array}
	\right.
\end{eqnarray}
where $C_2$ is a positive constant depending only on $\Omega$.

We will frequently use the following regularity theory of elliptic equations \cite{chen1998second,gao2017linearized} and discrete version of the Gronwall's lemma \cite{shen2011spectral}

\begin{lemma}\label{nele}
	Suppose that $\Omega\in \mathbb{R}^d$ $(d=2,3)$ be a bounded and smooth domain and $\phi\in H^k(\Omega)\cap L^2_0(\Omega)$ is a solution of 
	$$
	-\Delta\phi = f, ~ \bm{x}\in \Omega, \quad \nabla\phi\cdot\bm{n}|_{\partial\Omega} = 0,
	$$
	where $\int_{\Omega}fd\bm{x} = 0$. Then the following estimate holds for $1 < q < \infty$
	\begin{eqnarray}
		\|\phi\|_{W^{2,q}} \leq \|f\|_{L^q}.
	\end{eqnarray}
\end{lemma}

\begin{lemma}\label{le3.1}
	Let $a^n$, $b^n$, $c^n$, $d^n$ be four nonnegative sequences satisfying
	$$
	a^n+\tau\sum\limits^m_{n=0}b^n\leq B+\tau\sum\limits^m_{n=0}(a^nc^n+d^n) \quad \text{with} \quad
	\tau\sum\limits^{T/\tau}_{n=0}c^n\leq M,\ \forall 0 \leq m \leq T/\tau,
	$$
	where $B$ and $M$ are two positive constants depending on $h$ and $\tau$. We assume $\tau c^n<1$ and let $\sigma=\max\limits_{0\leq n \leq T/\tau}(1-\tau c^n)^{-1}$, then
	\begin{eqnarray}\label{hy3.6}
		a^m+\tau\sum\limits^m_{n=1}b^n\leq \exp(\sigma M)\left(B+\tau\sum\limits^m_{n=0}d^n\right),\quad
		\forall m \leq T/\tau.
	\end{eqnarray}
\end{lemma}

Finally, we provide the following results of uniform bounds which are direct consequence of the energy stability in Theorem \ref{th2.1}.
\begin{lemma}\label{le3.2}
Assume that the numerical solution $(c^{n+1}_1,c^{n+1}_2,\phi^{n+1},\bar{\bm{u}}^{n+1},\bm{u}^{n+1},p^{n+1},r^{n+1})$ satisfies the scheme \eqref{hy2.8} and Theorem \ref{th2.1}, then there exists a positive constant $C_3$ independent of $\tau$ such that  
\begin{eqnarray}
	\label{hy3.7} && \max_{0 \leq n \leq N-1}\left(\|u^{n+1}\|^2 + |r^{n+1}|^2 + \tau^2\|\nabla p^{n+1}\|^2\right) \leq C_3, \\
	\label{hy3.8} && \tau\sum_{n=0}^{N-1}\left(\|\nabla\bar{\bm{u}}^{n+1}\|^2 + \|\Delta\phi^{n+1}\|^2 + \|\sqrt{c^{n+1}_1+c^{n+1}_2}\nabla\phi\|^2\right) \leq C_3.
\end{eqnarray}
\end{lemma}

We set the temporal errors 
\begin{eqnarray*}
	\left\{
	\begin{array}{l}
		e^{n+1}_{c_1} = c_1(t^{n+1}) - c^{n+1}_1, \quad e^{n+1}_{c_2} = c_2(t^{n+1}) - c^{n+1}_2, \quad e^{n+1}_{\phi} = \phi(t^{n+1}) - \phi^{n+1}, \\
		e^{n+1}_{\bm{u}} = \bm{u}(t^{n+1}) - \bm{u}^{n+1}, \quad e^{n+1}_{\bar{\bm{u}}} = \bm{u}(t^{n+1}) - \bar{\bm{u}}^{n+1}, \quad e^{n+1}_p = p(t^{n+1}) - p^{n+1}, \\
		e^{n+1}_r = r(t^{n+1}) - r^{n+1}.
	\end{array}
\right.
\end{eqnarray*}
Then we have the following results of temporal error estimates with respect to the step size $\tau$.

\begin{theorem}\label{th3.3}
Give the initial values $c^0_1(\bm{x})$, $c^0_2(\bm{x})$ which are non-negative almost everywhere in $\Omega$ and assume that the exact solution $(c_1,c_2,\phi,\bm{u},p,r)$ of the system \eqref{hy2.6} satisfies the following regularity:
\begin{eqnarray}\label{hy3.9}
	\left\{
	\begin{array}{l}
		c_1,c_2 \in C(0,T;H^3(\Omega))\cap C^2(0,T;H^2(\Omega)), \quad \phi\in C(0,T;H^2(\Omega)), \\
		\bm{u} \in C(0,T;\bm{H}^2(\Omega))\cap C^2(0,T;\bm{L}^2(\Omega)), \quad p\in C^2(0,T;H^1(\Omega)), \quad r\in C^2([0,T]),
	\end{array}
	\right.
\end{eqnarray}
then there exists a sufficiently small positive constant $\tau_*$ such that when $\tau \leq \tau_*$, the scheme \eqref{hy2.8} satisfies the following error estimates
\begin{eqnarray}
	\label{hy3.10} \max_{0 \leq n \leq N-1}\left(\|e^{n+1}_{c_1}\|^2_{H^1} + \|e^{n+1}_{c_2}\|^2_{H^1} + \|e^{n+1}_{\phi}\|^2_{H^2} + \|e^{n+1}_{\bm{u}}\|^2 + |e^{n+1}_r|^2 + \tau^2\|e^{n+1}_p\|^2_{H^1}\right) \leq C\tau^2, \\
	\label{hy3.11}  \tau\sum_{n=0}^{N-1}\left(\|e^{n+1}_{c_1}\|^2_{H^2} + \|e^{n+1}_{c_2}\|^2_{H^2} + \|e^{n+1}_{\bm{u}}\|^2_{H^1} \right) \leq C\tau^2,
\end{eqnarray}
where $C$ is a positive constant independent of $\tau$.
\end{theorem}

We shall invoke a mathematical induction on
\begin{eqnarray}\label{hy3.12}
	\max_{0 \leq n \leq m}\left\{\|c^n_1\|_{L^{\infty}},\|c^n_2\|_{L^{\infty}},\|\bm{u}^n\|_{H^1}\right\} \leq C_N,
\end{eqnarray}
for some non-negative $m \leq N-1$, where $C_N$ denotes
$$
C_N := \max_{0 \leq t \leq T}\left\{\|c_1\|_{L^{\infty}},\|c_2\|_{L^{\infty}},\|\bm{u}\|_{H^1}\right\} + 1.
$$
When $n=0$, we have $\|c^0_1\|_{L^{\infty}} = \|c_1(t^0)\|_{L^{\infty}} \leq C_N$. Similarly, we get $\|c^0_2\|_{L^{\infty}} \leq C_N$ and $\|\bm{u}^0\|_{L^{\infty}} \leq C_N$. Then we shall see that if \eqref{hy3.12} holds for $0 \leq n \leq m$, then it also holds for $n=m+1$. This proof will be carried out with some sequences of lemmas below.

\subsection{Estimates for the ionic concentration and potential}\label{subse3.1}

We derive the $L^{\nu_k}$ estimates of the ionic concentrations $c^{n+1}_1$, $c^{n+1}_2$ for $1 \leq k \leq \infty$ and the $W^{2,\nu_k}$ estimate of potential function $\phi^{n+1}$ for $1 \leq k < \infty$, where we define $\nu_k := 2^k$. These play a key role in the error estimates.

\begin{lemma}\label{le3.4}
Assume that all the conditions in Theorem \ref{th3.3} hold, then under the induction assumption \eqref{hy3.12}, there exists a sufficiently small positive constant $\tau_0$ such that when $\tau \leq \tau_0$, we have
\begin{eqnarray}
	\label{hy3.13} && \max_{0 \leq n \leq m}\left(\|c^{n+1}_1\|_{L^{\nu_k}} + \|c^{n+1}_2\|_{L^{\nu_k}}\right) \leq C, \quad 1 \leq k \leq \infty, \\
	\label{hy3.14} && \max_{0 \leq n \leq m}\|\phi^{n+1}\|_{W^{2,\nu_k}} \leq C, \quad 1 \leq k < \infty,
\end{eqnarray}
where $\nu_k := 2^k$.
\end{lemma}

\begin{proof}
Taking the inner product of \eqref{hy2.8a} with $\tau(c^{n+1}_1)^{\nu_k-1}$ for $1 \leq k \leq \infty$, we have
\begin{eqnarray}\label{hy3.15}
	\|c^{n+1}_1\|^{\nu_k}_{L^{\nu_k}} \leq \left(c^n_1,(c^{n+1}_1)^{\nu_k-1}\right) -\tau\left(c^{n+1}_1\nabla\phi^n,\nabla(c^{n+1}_1)^{\nu_k-1}\right),
\end{eqnarray}
where the term $(\bm{u}\cdot\nabla c^{n+1}_1,(c^{n+1}_1)^{\nu_k-1}) = 0$ and we use
$$
\left(\nabla c^{n+1}_1,\nabla(c^{n+1}_1)^{\nu_k-1}\right) = (\nu_k-1)\left(\nabla c^{n+1}_1,(c^{n+1}_1)^{\nu_k-2}\nabla c^{n+1}_1\right) \geq 0.
$$
Using the H\"older's inequality, we get
\begin{eqnarray}\label{hy3.16}
	\left(c^n_1,(c^{n+1}_1)^{\nu_k-1}\right) \leq \|c^n_1\|_{L^{\nu_k}}\|c^{n+1}_1\|^{\nu_k-1}_{L^{\nu_k}}. 
\end{eqnarray}
By using \eqref{hy2.8c} and \eqref{hy3.12}, we have
\begin{eqnarray}\label{hy3.17}
	\left(c^{n+1}_1\nabla\phi^n,\nabla(c^{n+1}_1)^{\nu_k-1}\right) &=& \frac{\nu_k-1}{\nu_k}\left(\nabla\phi^n,\nabla(c^{n+1}_1)^{\nu_k}\right) = \frac{\nu_k-1}{\nu_k}\left(c^n_1 - c^n_2,(c^{n+1}_1)^{\nu_k}\right) \nonumber\\
	&\leq& \frac{\nu_k-1}{\nu_k}\|c^n_1 - c^n_2\|_{L^{\infty}}\|c^{n+1}_1\|^{\nu_k}_{L^{\nu_k}} \leq \frac{2C_N(\nu_k-1)}{\nu_k}\|c^{n+1}_1\|^{\nu_k}_{L^{\nu_k}}.
\end{eqnarray}
Then together with \eqref{hy3.15}-\eqref{hy3.17}, we have
\begin{eqnarray}\label{hy3.18}
	\|c^{n+1}_1\|_{L^{\nu_k}} - \|c^n_1\|_{L^{\nu_k}} \leq \frac{2\tau C_N(\nu_k-1)}{\nu_k}\|c^{n+1}_1\|_{L^{\nu_k}}.
\end{eqnarray}
Hence for the sufficiently small constant $\tau_0>0$ satisfying $2\tau_0C_N(\nu_k-1) / \nu_k \leq 1$ such that when $\tau \leq \tau_0$, we derive from \eqref{hy3.18} that 
\begin{eqnarray}\label{hy3.19}
	\max_{0 \leq n \leq m}\|c^{n+1}_1\|_{L^{\nu_k}} \leq C, \quad 1 \leq k \leq \infty.
\end{eqnarray}
Similarly, taking the inner product of \eqref{hy2.8b} with $\tau(c^{n+1}_2)^{\nu_k-1}$ for $1 \leq k \leq \infty$, we have
\begin{eqnarray}\label{hy3.20}
	\max_{0 \leq n \leq m}\|c^{n+1}_2\|_{L^{\nu_k}} \leq C, \quad 1 \leq k \leq \infty.
\end{eqnarray}
Combining \eqref{hy3.19} with \eqref{hy3.20}, it leads to the desired result \eqref{hy3.13}.

By using Lemma \ref{nele} with case $d=2$, we derive from \eqref{hy2.8c} and \eqref{hy3.13} that
\begin{eqnarray}\label{hy3.21}
	\max_{0 \leq n \leq m}\|\phi^{n+1}\|_{W^{2,\nu_k}} \leq C\max_{0 \leq n \leq m}\left(\|c^{n+1}_1\|_{L^{\nu_k}} + \|c^{n+1}_2\|_{L^{\nu_k}}\right) \leq C, \quad 1 \leq k < \infty,
\end{eqnarray}
which implies the desired result \eqref{hy3.14}.
\end{proof}

\subsection{Error estimates for the ionic concentration, potential and velocity}\label{subse3.2}

We derive the following $L^2$ error estimates for the ionic concentrations $c_1$, $c_2$, velocity $\bm{u}$ and $H^2$ error estimate for the potential $\phi$.

\begin{lemma}\label{le3.5}
	Assume that all the conditions in Theorem \ref{th3.3} hold, then under the induction assumption \eqref{hy3.12}, there exists a sufficiently small positive constant $\tau_1 \leq \tau_0$ such that when $\tau \leq \tau_1$, we have
	\begin{eqnarray}\label{hy3.22} 
		\begin{aligned}
		\max_{0 \leq n \leq m}\big(\|e^{n+1}_{c_1}\|^2 + \|e^{n+1}_{c_2}\|^2 + \| e^{n+1}_{\phi}\|^2_{H^2}  + \|e^{n+1}_{\bm{u}}\|^2 + \tau^2\|e^{n+1}_p\|^2_{H^1} + |e^{n+1}_r|^2 \big) \\
		+ \tau\sum_{n=0}^{m}\big(\|e^{n+1}_{c_1}\|^2_{H^1} + \| e^{n+1}_{c_2}\|^2_{H^1}  + \|e^{n+1}_{\bar{\bm{u}}}\|^2_{H^1} + \| e^{n+1}_{\bm{u}}\|^2_{H^1} \big) \leq C\tau^2.
	\end{aligned}
	\end{eqnarray}
\end{lemma}

\begin{proof}
We shall follow five steps.

{\bf Step 1.} Subtracting \eqref{hy2.8c} at $t^{n+1}$ from \eqref{hy2.6c}, we have
\begin{eqnarray}\label{hy3.23}
	-\Delta e^{n+1}_{\phi} = e^{n+1}_{c_1} - e^{n+1}_{c_2}.
\end{eqnarray}
Taking the inner products of \eqref{hy3.23} with $e^{n+1}_{\phi}$ and $-\Delta e^{n+1}_{\phi}$, we have
\begin{eqnarray*}
	 \|\nabla e^{n+1}_{\phi}\|^2 \leq \left(\|e^{n+1}_{c_1}\| + \|e^{n+1}_{c_2}\|\right)\|\nabla e^{n+1}_{\phi}\|, \quad \|\Delta e^{n+1}_{\phi}\|^2 \leq \left(\|e^{n+1}_{c_1}\| + \|e^{n+1}_{c_2}\|\right)\|\Delta e^{n+1}_{\phi}\|,
\end{eqnarray*}
which imply that
\begin{eqnarray}\label{hy3.24}
	\|\nabla e^{n+1}_{\phi}\|^2 + \|\Delta e^{n+1}_{\phi}\|^2 \leq C\left(\|e^{n+1}_{c_1}\|^2 + \|e^{n+1}_{c_2}\|^2\right).
\end{eqnarray}

{\bf Step 2.} Subtracting \eqref{hy2.8a} and \eqref{hy2.8b} at $t^{n+1}$ from \eqref{hy2.6a} and \eqref{hy2.6b} respectively, we have
\begin{eqnarray}
	\label{hy3.25} 
	\begin{aligned}
		 \delta_t e^{n+1}_{c_1} - \Delta e^{n+1}_{c_1} + \bm{u}^n\cdot\nabla e^{n+1}_{c_1} + e^n_{\bm{u}}\cdot\nabla c_1(t^{n+1}) - \nabla\cdot\left(c^{n+1}_1\nabla e^n_{\phi}\right) \\
		 - \nabla\cdot\left(e^{n+1}_{c_1}\nabla\phi(t^n)\right) = R^{n+1}_{c_1},
	\end{aligned} \\
	\label{hy3.26} 
	\begin{aligned}
		\delta_t e^{n+1}_{c_2} - \Delta e^{n+1}_{c_2} + \bm{u}^n\cdot\nabla e^{n+1}_{c_2} + e^n_{\bm{u}}\cdot\nabla c_2(t^{n+1}) + \nabla\cdot\left(c^{n+1}_2\nabla e^n_{\phi}\right) \\
		+ \nabla\cdot\left(e^{n+1}_{c_2}\nabla\phi(t^n)\right) = R^{n+1}_{c_2},
	\end{aligned}
\end{eqnarray}
where the truncation errors $R^{n+1}_{c_1}$ and $R^{n+1}_{c_2}$ denote
\begin{eqnarray*}
	R^{n+1}_{c_1} &=& \partial_tc_1(t^{n+1}) - \delta_tc_1(t^{n+1}) + \left(\bm{u}(t^n) - \bm{u}(t^{n+1})\right)\cdot\nabla c_1(t^{n+1}) \\
	&& + \nabla\cdot\left(c_1(t^{n+1})\nabla(\phi(t^n) - \phi(t^{n+1}))\right), \\
	R^{n+1}_{c_2} &=& \partial_tc_2(t^{n+1}) - \delta_tc_2(t^{n+1}) + \left(\bm{u}(t^n) - \bm{u}(t^{n+1})\right)\cdot\nabla c_2(t^{n+1}) \\
	&& + \nabla\cdot\left(c_2(t^{n+1})\nabla(\phi(t^n) - \phi(t^{n+1}))\right),
\end{eqnarray*}
and satisfy
$$
\max_{0 \leq n \leq N-1}\left(\|R^{n+1}_{c_1}\|^2 + \|R^{n+1}_{c_2}\|^2\right) + \tau\sum_{n=0}^{N-1}\left(\|\delta_tR^{n+1}_{c_1}\|^2 + \|\delta_tR^{n+1}_{c_2}\|^2\right) \leq C\tau^2.
$$
Taking the inner product of \eqref{hy3.25} with $e^{n+1}_{c_1}$, we have
\begin{eqnarray}\label{hy3.27}
	\begin{aligned}
		& \frac{1}{2\tau}\left(\|e^{n+1}_{c_1}\|^2 - \|e^{n}_{c_1}\|^2 + \|e^{n+1}_{c_1} - e^n_{c_1}\|^2\right) + \|\nabla e^{n+1}_{c_1}\|^2 \\
		=& -(e^n_{\bm{u}}\cdot\nabla c_1(t^{n+1}),e^{n+1}_{c_1}) - (c^{n+1}_1\nabla e^n_{\phi},\nabla e^{n+1}_{c_1}) \\
		& - (e^{n+1}_{c_1}\nabla\phi(t^n), \nabla e^{n+1}_{c_1}) - (R^{n+1}_{c_1}, e^{n+1}_{c_1}) =: \sum_{i=1}^{4}I^{n+1}_i,
	\end{aligned}
\end{eqnarray}
where the term $(\bm{u}^n\cdot\nabla e^{n+1}_{c_1},e^{n+1}_{c_1}) = 0$. For the terms $I^{n+1}_1$ and $I^{n+1}_4$, we directly obtain
\begin{eqnarray*}
	&& I^{n+1}_1 \leq \|e^n_{\bm{u}}\|\|\nabla c_1(t^{n+1})\|_{L^{\infty}}\|e^{n+1}_{c_1}\| \leq C\epsilon^{-1}\|e^n_{\bm{u}}\|^2 + \epsilon\|\nabla e^{n+1}_{c_1}\|^2, \\
	&& I^{n+1}_4 \leq C\epsilon^{-1}\|R^{n+1}_{c_1}\|^2_{H^{-1}} + \epsilon\|\nabla e^{n+1}_{c_1}\|^2.
\end{eqnarray*}
For the terms $I^{n+1}_2$ and $I^{n+1}_3$, we use \eqref{hy3.24} to obtain 
\begin{eqnarray*}
	I^{n+1}_2 + I^{n+1}_3 &\leq& \left(\|c^{n+1}_1\|_{L^{\infty}}\|\nabla e^n_{\phi}\| + \|e^{n+1}_{c_1}\|\|\nabla\phi(t^n)\|_{L^{\infty}}\right)\|\nabla e^{n+1}_{c_1}\| \\
	&\leq& C\epsilon^{-1}\left(\|e^{n+1}_{c_1}\|^2 + \|e^n_{c_1}\|^2 + \|e^n_{c_2}\|^2\right) + \epsilon\|\nabla e^{n+1}_{c_1}\|^2.
\end{eqnarray*}
Then by choosing $\epsilon=1/6$, \eqref{hy3.27} can be rewritten as 
\begin{eqnarray}\label{hy3.28}
	\begin{aligned}
	 & \frac{1}{2\tau}\left(\|e^{n+1}_{c_1}\|^2 - \|e^{n}_{c_1}\|^2 + \|e^{n+1}_{c_1} - e^{n}_{c_1}\|^2\right) + \frac{1}{2}\|\nabla e^{n+1}_{c_1}\|^2 \\
	\leq& C\left(\|e^{n+1}_{c_1}\|^2 + \|e^n_{c_1}\|^2 + \|e^n_{c_2}\|^2 + \|e^n_{\bm{u}}\|^2 + \|R^{n+1}_{c_1}\|^2_{H^{-1}}\right).
	\end{aligned}
\end{eqnarray}
Similarly, taking the inner product of \eqref{hy3.26} with $e^{n+1}_{c_2}$, we have
\begin{eqnarray}\label{hy3.29}
	\begin{aligned}
	& \frac{1}{2\tau}\left(\|e^{n+1}_{c_2}\|^2 - \|e^{n}_{c_2}\|^2 + \|e^{n+1}_{c_2} - e^{n}_{c_2}\|^2\right) + \frac{1}{2}\|\nabla e^{n+1}_{c_2}\|^2 \\
	\leq& C\left(\|e^{n+1}_{c_2}\|^2 + \|e^n_{c_1}\|^2 + \|e^n_{c_2}\|^2 + \|e^n_{\bm{u}}\|^2 + \|R^{n+1}_{c_2}\|^2_{H^{-1}}\right).
	\end{aligned}
\end{eqnarray}

{\bf Step 3.} Subtracting \eqref{hy2.8d} at $t^{n+1}$ from \eqref{hy2.6d}, we have
\begin{eqnarray}\label{hy3.30} 
	\begin{aligned}
	& \frac{e^{n+1}_{\bar{\bm{u}}} - e^n_{\bm{u}}}{\tau} - \Delta e^{n+1}_{\bar{\bm{u}}} + \nabla e^n_p + \frac{e^{n+1}_r}{\sqrt{E(\phi^{n+1})}} \big(\bm{u}^n\cdot\nabla\bm{u}^n + (c^n_1-c^n_2)\nabla\phi^n\big) \\
	=& r(t^{n+1})\Big(\frac{1}{\sqrt{E(\phi^{n+1})}} - \frac{1}{\sqrt{E(\phi(t^{n+1}))}}\Big) \big(\bm{u}(t^n)\cdot\nabla\bm{u}(t^n) + (c_1(t^n) - c_2(t^n))\nabla\phi(t^n)\big) \\
	& - \frac{r(t^{n+1})}{\sqrt{E(\phi^{n+1})}} \big(\bm{u}(t^n)\cdot\nabla\bm{u}(t^n) - \bm{u}^n\cdot\nabla\bm{u}^n \big) \\
	& - \frac{r(t^{n+1})}{\sqrt{E(\phi^{n+1})}} \big( (c_1(t^n) - c_2(t^n))\nabla\phi(t^n) - (c^n_1 - c^n_2)\nabla\phi^n\big) + R^{n+1}_{\bm{u}}\\
	=:& \sum_{i=1}^{4}J^{n+1}_i,
	\end{aligned}
\end{eqnarray}
where the truncation error $R^{n+1}_{\bm{u}}$ satisfies
\begin{eqnarray*}
	R^{n+1}_{\bm{u}} &=& \partial_t\bm{u}(t^{n+1}) - \delta_t\bm{u}(t^{n+1}) + \nabla p(t^n) - \nabla p(t^{n+1})
	+ \bm{u}(t^n)\cdot\nabla\bm{u}(t^n) - \bm{u}(t^{n+1})\cdot\nabla\bm{u}(t^{n+1})  \\
	&& + (c_1(t^n) - c_2(t^n))\nabla\phi(t^n) - (c_1(t^{n+1}) - c_2(t^{n+1}))\nabla\phi(t^{n+1}),
\end{eqnarray*}
and satisfies
$$
\max_{0 \leq n \leq N-1}\|R^{n+1}_{\bm{u}}\|^2 + \tau\sum_{n=0}^{N-1}\|\delta_tR^{n+1}_{\bm{u}}\|^2 \leq C\tau^2.
$$
Taking the inner product of \eqref{hy3.30} with $e^{n+1}_{\bar{\bm{u}}}$, we have
\begin{eqnarray}\label{hy3.31}
	\begin{aligned}
	& \frac{1}{2\tau}\left(\|e^{n+1}_{\bar{\bm{u}}}\|^2 - \|e^n_{\bm{u}}\|^2 + \|e^{n+1}_{\bar{\bm{u}}} - e^n_{\bm{u}}\|^2\right) + \|\nabla e^{n+1}_{\bar{\bm{u}}}\|^2 + (\nabla e^n_p, e^{n+1}_{\bar{\bm{u}}}) \\
	=& - \frac{e^{n+1}_r}{\sqrt{E(\phi^{n+1})}} \big[(\bm{u}^n\cdot\nabla\bm{u}^n, e^{n+1}_{\bar{\bm{u}}}) + ((c^n_1-c^n_2)\nabla\phi^n, e^{n+1}_{\bar{\bm{u}}})\big] + \sum_{i=1}^{4}(J^{n+1}_i, e^{n+1}_{\bar{\bm{u}}}).
	\end{aligned}
\end{eqnarray}
It follows from \eqref{hy3.7} and \eqref{hy3.14} that $r^{n+1}/\sqrt{E(\phi^{n+1})} \leq C_3/\sqrt{C_0} \leq C$ and $1/\sqrt{E(\phi^{n+1})} \leq 1/\sqrt{C_0} \leq C$, we use \eqref{hy3.14} and \eqref{hy3.24} to obtain 
\begin{eqnarray}\label{hy3.32}
	\begin{aligned}
	\frac{1}{\sqrt{E(\phi^{n+1})}} - \frac{1}{\sqrt{E(\phi(t^{n+1}))}} =& \frac{E(\phi(t^{n+1})) - E(\phi^{n+1})}{\sqrt{E(\phi(t^{n+1}))E(\phi^{n+1})(E(\phi(t^{n+1})) + E(\phi^{n+1}))}} \\
	\leq& C\|\nabla e^{n+1}_{\phi}\| \left(\|\nabla\phi(t^{n+1})\| + \|\nabla\phi^{n+1}\|\right) \\
	\leq& C\left(\|e^{n+1}_{c_1}\| + \|e^{n+1}_{c_2}\|\right).
	\end{aligned}
\end{eqnarray}
Then we can estimate the term $(J^{n+1}_1,e^{n+1}_{\bar{\bm{u}}})$ as follows
\begin{eqnarray*}
	(J^{n+1}_1,e^{n+1}_{\bar{\bm{u}}}) &\leq& C|r(t^{n+1})|\left(\|e^{n+1}_{c_1}\| + \|e^{n+1}_{c_2}\|\right)\big(\|\bm{u}(t^n)\|_{L^{\infty}}\|\nabla\bm{u}(t^n)\| \nonumber \\
	&& + \|c_1(t^n)-c_2(t^n)\|_{L^{\infty}}\|\nabla\phi(t^n)\|\big) \|e^{n+1}_{\bar{\bm{u}}}\| \nonumber \\
	&\leq& C\epsilon^{-1}\left(\|e^{n+1}_{c_1}\|^2 + \|e^{n+1}_{c_2}\|^2\right) + \epsilon\|\nabla e^{n+1}_{\bar{\bm{u}}}\|^2.
\end{eqnarray*}
By using \eqref{hy3.2}, \eqref{hy3.5} and \eqref{hy3.12}, we have
\begin{eqnarray*}
	(J^{n+1}_2,e^{n+1}_{\bar{\bm{u}}}) &=& -\frac{r(t^{n+1})}{\sqrt{E(\phi^{n+1})}} \Big[(e^n_{\bm{u}}\cdot\nabla\bm{u}(t^n) + \bm{u}(t^n)\cdot\nabla e^n_{\bm{u}}, e^{n+1}_{\bar{\bm{u}}}) - (e^n_{\bm{u}}\cdot\nabla e^n_{\bm{u}}, e^{n+1}_{\bar{\bm{u}}})\Big] \nonumber\\
	&\leq& C\big( \|e^n_{\bm{u}}\|\|\bm{u}(t^n)\|_{H^2}\|e^{n+1}_{\bar{\bm{u}}}\|_{H^1} + \|e^n_{\bm{u}}\|^{\frac{1}{2}} \|e^n_{\bm{u}}\|^{\frac{1}{2}}_{H^1} \|e^n_{\bm{u}}\|^{\frac{1}{2}} \|e^n_{\bm{u}}\|^{\frac{1}{2}}_{H^1} \|e^{n+1}_{\bar{\bm{u}}}\|_{H^1}\big) \nonumber \\
	&\leq& C\epsilon^{-1}\|e^n_{\bm{u}}\|^2 + \epsilon\|\nabla e^{n+1}_{\bar{\bm{u}}}\|^2,
\end{eqnarray*}
where we use
$$
\max_{0 \leq n \leq m}\|e^n_{\bm{u}}\|_{H^1} \leq \max_{0 \leq n \leq m} \big(\|\bm{u}^n\|_{H^1} + \|\bm{u}(t^n)\|_{H^1}\big) \leq 2C_N.
$$
By using \eqref{hy3.12} and \eqref{hy3.24}, we have
\begin{eqnarray*}
	(J^{n+1}_3,e^{n+1}_{\bar{\bm{u}}}) &=& -\frac{r(t^{n+1})}{\sqrt{E(\phi^{n+1})}} \Big[\big((e^n_{c_1} - e^n_{c_2})\nabla\phi(t^n), e^{n+1}_{\bar{\bm{u}}}\big) + \big((c^n_1 - c^n_2)\nabla e^n_{\phi}, e^{n+1}_{\bar{\bm{u}}}\big)\Big] \nonumber \\
	&\leq& C\big(\|e^n_{c_1} - e^n_{c_2}\|\|\nabla\phi(t^n)\|_{L^{\infty}} + \|c^n_1-c^n_2\|_{L^{\infty}}\|\nabla e^n_{\phi}\|\big) \|e^{n+1}_{\bar{\bm{u}}}\| \nonumber \\
	&\leq& C\epsilon^{-1}\big(\|e^n_{c_1}\|^2 + \|e^n_{c_2}\|^2\big) + \epsilon\|\nabla e^{n+1}_{\bar{\bm{u}}}\|^2,
\end{eqnarray*}
and 
\begin{eqnarray*}
	(J^{n+1}_4,e^{n+1}_{\bar{\bm{u}}}) \leq C\epsilon^{-1}\|R^{n+1}_{\bm{u}}\|^2_{H^{-1}} + \epsilon\|\nabla e^{n+1}_{\bar{\bm{u}}}\|^2.
\end{eqnarray*}

Next we estimate the term $(\nabla e^n_p, e^{n+1}_{\bar{\bm{u}}})$. From \eqref{hy2.8e} and \eqref{hy2.8f}, we obtain
\begin{eqnarray}
	\label{hy3.33} && \frac{e^{n+1}_{\bm{u}} - e^{n+1}_{\bar{\bm{u}}}}{\tau} + \nabla e^{n+1}_p - \nabla e^n_p = R^{n+1}_p, \\
	\label{hy3.34} && \nabla\cdot e^{n+1}_{\bm{u}} = 0,
\end{eqnarray}
where the truncation error $R^{n+1}_p$ denotes
$$
R^{n+1}_p = \nabla p(t^n) - \nabla p(t^{n+1}),
$$
and satisfies
$$
\max_{0 \leq n \leq N-1}\|R^{n+1}_p\|^2 + \tau\sum_{n=0}^{N-1}\|\delta_tR^{n+1}_p\|^2 \leq C\tau^2.
$$
We rewrite \eqref{hy3.33} as
\begin{eqnarray}\label{hy3.35}
	e^{n+1}_{\bm{u}} + \tau\nabla e^{n+1}_p = e^{n+1}_{\bar{\bm{u}}} + \tau\nabla e^n_p + \tau R^{n+1}_p.
\end{eqnarray}
Taking the inner product of \eqref{hy3.35} with itself on both sides, we have
\begin{eqnarray}\label{hy3.36}
	\begin{aligned}
	\|e^{n+1}_{\bm{u}} \|^2 + \tau^2\|\nabla e^{n+1}_p\|^2 = \|e^{n+1}_{\bar{\bm{u}}}\|^2 + \tau^2\|\nabla e^n_p\|^2 + \tau^2\|R^{n+1}_p\|^2  \\
	+ 2\tau(\nabla e^n_p,e^{n+1}_{\bar{\bm{u}}}) + 2\tau(e^{n+1}_{\bar{\bm{u}}}, R^{n+1}_p) + 2\tau^2(\nabla e^n_p, R^{n+1}_p). 
	\end{aligned}
\end{eqnarray}
The terms $\tau(\nabla e^n_p, R^{n+1}_p)$ and $(e^{n+1}_{\bar{\bm{u}}},R^{n+1}_p)$ can be estimated as follows
\begin{eqnarray*}
	&& \tau(\nabla e^n_p, R^{n+1}_p) \leq C\big(\tau^2\|\nabla e^n_p\|^2 + \|R^{n+1}_p\|^2\big), \\
	&& (e^{n+1}_{\bar{\bm{u}}},R^{n+1}_p) \leq C\epsilon^{-1}\|R^{n+1}_p\|^2_{H^{-1}} + \epsilon\|\nabla e^{n+1}_{\bar{\bm{u}}}\|^2.
\end{eqnarray*}

By choosing $\epsilon=1/10$ and together with \eqref{hy3.31} and \eqref{hy3.36}, we obtain
\begin{eqnarray}\label{hy3.37}
	\begin{aligned}
	& \frac{1}{2\tau}\left(\|e^{n+1}_{\bm{u}}\|^2 - \|e^n_{\bm{u}}\|^2 + \|e^{n+1}_{\bar{\bm{u}}} - e^n_{\bm{u}}\|^2\right) + \frac{\tau}{2}\left(\|\nabla e^{n+1}_p\|^2 - \|\nabla e^n_p\|^2\right) + \frac{1}{2}\|\nabla e^{n+1}_{\bar{\bm{u}}}\|^2 \\
	\leq& - \frac{e^{n+1}_r}{\sqrt{E(\phi^{n+1})}} \big[(\bm{u}^n\cdot\nabla\bm{u}^n, e^{n+1}_{\bar{\bm{u}}}) + ((c^n_1-c^n_2)\nabla\phi^n, e^{n+1}_{\bar{\bm{u}}})\big] + C\big( \|e^{n+1}_{c_1}\|^2 + \|e^{n+1}_{c_2}\|^2 \\
	& + \|e^{n}_{c_1}\|^2 + \|e^{n}_{c_2}\|^2 + \|e^n_{\bm{u}}\|^2 + \tau^2\|\nabla e^n_p\|^2 + \|R^{n+1}_p\|^2 + \|R^{n+1}_p\|^2_{H^{-1}} + \|R^{n+1}_{\bm{u}}\|^2_{H^{-1}} \big).
	\end{aligned}
\end{eqnarray}

{\bf Step 4.} Subtracting \eqref{hy2.8g} at $t^{n+1}$ from \eqref{hy2.6f}, we have
\begin{eqnarray}\label{hy3.38}
	\begin{aligned}
	& \delta_te^{n+1}_r - \frac{1}{2\sqrt{E(\phi^{n+1})}} \big[(\bm{u}^n\cdot\nabla\bm{u}^n, e^{n+1}_{\bar{\bm{u}}}) + ((c^n_1-c^n_2)\nabla\phi^n, e^{n+1}_{\bar{\bm{u}}})\big]  \\
	=& \frac{1}{2}\Big(\frac{1}{\sqrt{E(\phi^{n+1})}} - \frac{1}{\sqrt{E(\phi(t^{n+1}))}}\Big) \big(\bm{u}(t^n)\cdot\nabla\bm{u}(t^n) + (c_1(t^n) - c_2(t^n))\nabla\phi(t^n), \bm{u}(t^{n+1})\big) \\
	& + \frac{1}{2\sqrt{E(\phi^{n+1})}} \big(\bm{u}(t^n)\cdot\nabla\bm{u}(t^n) - \bm{u}^n\cdot\nabla\bm{u}^n, \bm{u}(t^{n+1})\big) \\
	& + \frac{1}{2\sqrt{E(\phi^{n+1})}} \big( (c_1(t^n) - c_2(t^n))\nabla\phi(t^n) - (c_1^n - c_2^n)\nabla\phi^n, \bm{u}(t^{n+1})\big)  \\
	&   + \frac{e^{n+1}_r}{2r(t^{n+1})r^{n+1}} \Big( \|c_1(t^{n+1}) - c_2(t^{n+1})\|^2 + \|\sqrt{c_1(t^{n+1}) + c_2(t^{n+1})}\nabla\phi(t^{n+1})\|^2 \Big)  \\
	&  + \frac{1}{2r^{n+1}} \Big( \|c^{n+1}_1 - c^{n+1}_2\|^2 - \|c_1(t^{n+1}) - c_2(t^{n+1})\|^2 \Big)  \\
	& + \frac{1}{2r^{n+1}} \Big( \|\sqrt{c^{n+1}_1 + c^{n+1}_2}\nabla\phi^{n+1}\|^2 - \|\sqrt{c_1(t^{n+1}) + c_2(t^{n+1})}\nabla\phi(t^{n+1})\|^2 \Big) + R^{n+1}_r  \\
	=:& \sum_{i=1}^{6}K^{n+1}_i + R^{n+1}_r,
	\end{aligned}
\end{eqnarray}
where the truncation error $R^{n+1}_r$ denotes
\begin{eqnarray*}
	R^{n+1}_r &=& \partial_tr(t^{n+1}) - \delta_tr(t^{n+1}) + \frac{1}{2\sqrt{E(\phi(t^{n+1}))}} \big(\bm{u}(t^n)\cdot\nabla\bm{u}(t^n) - \bm{u}(t^{n+1})\cdot\nabla\bm{u}(t^{n+1}), \bm{u}(t^{n+1})\big) \\
	&& + \frac{1}{2\sqrt{E(\phi(t^{n+1}))}} \big((c_1(t^n) - c_2(t^n))\nabla\phi(t^n) - (c_1(t^{n+1}) - c_2(t^{n+1}))\nabla\phi(t^{n+1}), \bm{u}(t^{n+1})\big),
\end{eqnarray*}
and satisfies
$$
\max_{0 \leq n \leq N-1}|R^{n+1}_r|^2 \leq C\tau^2.
$$
From \eqref{hy3.32} and using \eqref{hy3.5}, we can estimate the term $K^{n+1}_1$ as follows
\begin{eqnarray*}
	K^{n+1}_1 &\leq& C\left(\|e^{n+1}_{c_1}\| + \|e^{n+1}_{c_2}\|\right)\big( \|\bm{u}(t^n)\|_{L^{\infty}}\|\nabla\bm{u}(t^n)\| + \|c_1(t^n)-c_2(t^n)\|_{L^{\infty}}\|\nabla\phi(t^n)\| \big)\|\bm{u}(t^{n+1})\| \\
	&\leq& C\left(\|e^{n+1}_{c_1}\| + \|e^{n+1}_{c_2}\|\right).
\end{eqnarray*}
By using \eqref{hy3.5}, we can estimate the term $K^{n+1}_2$ as follows
\begin{eqnarray*}
	K^{n+1}_2 &=& \frac{1}{2\sqrt{E(\phi^{n+1})}}\big(e^n_{\bm{u}}\cdot\nabla\bm{u}(t^n) + \bm{u}(t^n)\cdot\nabla e^n_{\bm{u}} - e^n_{\bm{u}}\cdot\nabla e^n_{\bm{u}}, \bm{u}(t^{n+1})\big) \\
	&\leq& C\big(\|\bm{u}(t^n)\|_{H^2} + \|e^n_{\bm{u}}\|_{H^1}\big) \|e^n_{\bm{u}}\|\|\bm{u}(t^{n+1})\|_{H^1} 
	\leq C\|e^n_{\bm{u}}\|.
\end{eqnarray*}
By using \eqref{hy3.12}, \eqref{hy3.13} with $k=1,\infty$ and \eqref{hy3.24}, we can estimate the terms $K^{n+1}_i$ $(i=3,4,5,6)$ as follows
\begin{eqnarray*}
	K^{n+1}_3 &=& \big((e^n_{c_1} - e^n_{c_2})\nabla\phi(t^n) + (c^n_1 - c^n_2)\nabla e^n_{\phi}, \bm{u}(t^{n+1})\big) \\
	&\leq& \big(\|e^n_{c_1} - e^n_{c_2}\|\|\nabla\phi(t^n)\| + \|c^n_1 - c^n_2\|\|\nabla e^n_{\phi}\|\big)\|\bm{u}(t^{n+1})\|_{L^{\infty}} \leq C\big(\|e^n_{c_1}\| + \|e^n_{c_2}\|\big), \\
	K^{n+1}_4 &\leq& C|e^{n+1}_r|, \\
	K^{n+1}_5 &\leq& C\big(\|e^{n+1}_{c_1}\| + \|e^{n+1}_{c_2}\|\big) \big(\|c^{n+1}_1\| + \|c^{n+1}_2\| + \|c_1(t^{n+1})\| + \|c_2(t^{n+1})\| \big) \\
	&\leq& C\big(\|e^{n+1}_{c_1}\| + \|e^{n+1}_{c_2}\|\big), \\
	K^{n+1}_6 &\leq& C\big(\|e^{n+1}_{c_1}\| + \|e^{n+1}_{c_2}\|\big)\|\nabla\phi(t^{n+1})\|^2_{L^4} + \|c^{n+1}_1 + c^{n+1}_2\|_{L^{\infty}}\|\nabla e^{n+1}_{\phi}\| \|\nabla\phi(t^{n+1}) + \nabla\phi^{n+1}\| \\
	&\leq& C\big(\|e^{n+1}_{c_1}\| + \|e^{n+1}_{c_2}\|\big).
\end{eqnarray*}
Then multiplying \eqref{hy3.38} with $2e^{n+1}_r$, we obtain
\begin{eqnarray}\label{hy3.39}
	\begin{aligned}
	&\frac{1}{\tau}\big(|e^{n+1}_r|^2 - |e^n_r|^2 + |e^{n+1}_r - e^n_r|^2\big) \\
	\leq& \frac{e^{n+1}_r}{\sqrt{E(\phi^{n+1})}} \big[(\bm{u}^n\cdot\nabla\bm{u}^n, e^{n+1}_{\bar{\bm{u}}}) + ((c^n_1-c^n_2)\nabla\phi^n, e^{n+1}_{\bar{\bm{u}}})\big] + C\big( \|e^{n+1}_{c_1}\|^2 + \|e^{n+1}_{c_2}\|^2 \\
	& + \|e^{n}_{c_1}\|^2 + \|e^{n}_{c_2}\|^2 + \|e^n_{\bm{u}}\|^2 + |e^{n+1}_r|^2 + |R^{n+1}_r|^2 \big).
	\end{aligned}
\end{eqnarray}

{\bf Step 5.} Together with \eqref{hy3.28}, \eqref{hy3.29}, \eqref{hy3.37} and \eqref{hy3.39}, we have
\begin{eqnarray}\label{hy3.40}
	\begin{aligned}
	&\|e^{n+1}_{c_1}\|^2 - \|e^n_{c_1}\|^2 + \|e^{n+1}_{c_2}\|^2 - \|e^n_{c_2}\|^2 + \|e^{n+1}_{\bm{u}}\|^2 - \|e^{n}_{\bm{u}}\|^2 + \tau^2(\|\nabla e^{n+1}_p\|^2 - \|\nabla e^n_p\|^2) \\
	& \quad + 2(|e^{n+1}_r|^2 - |e^n_r|^2) + \tau\|\nabla e^{n+1}_{c_1}\|^2 + \tau\|\nabla e^{n+1}_{c_2}\|^2 + \tau\|\nabla e^{n+1}_{\bar{\bm{u}}}\|^2 \\
	\leq& C\tau \big( \|e^{n+1}_{c_1}\|^2 + \|e^{n+1}_{c_2}\|^2 + \|e^{n}_{c_1}\|^2 + \|e^{n}_{c_2}\|^2 + \|e^n_{\bm{u}}\|^2 + \tau^2\|\nabla e^n_p\|^2 + |e^{n+1}_r|^2 \\
	& + \|R^{n+1}_{c_1}\|^2_{H^{-1}} + \|R^{n+1}_{c_2}\|^2_{H^{-1}} + \|R^{n+1}_{\bm{u}}\|^2_{H^{-1}} + \|R^{n+1}_p\|^2_{H^{-1}} + \|R^{n+1}_p\|^2 + |R^{n+1}_r|^2\big).
	\end{aligned}
\end{eqnarray}
Summing up \eqref{hy3.40} for $n$ from 0 to $\ell$ with $\ell\leq m$, we obtain
\begin{eqnarray}\label{hy3.41}
	\begin{aligned}
	& \max_{0 \leq n \leq \ell}\big( \|e^{n+1}_{c_1}\|^2 + \|e^{n+1}_{c_2}\|^2 + \|e^{n+1}_{\bm{u}}\|^2 + \tau^2\|\nabla e^{n+1}_p\|^2 + |e^{n+1}_r|^2 \big)  \\
	& \quad + \tau\sum_{n=0}^{\ell}\big(\|\nabla e^{n+1}_{c_1}\|^2 + \|\nabla e^{n+1}_{c_2}\|^2 + \|\nabla e^{n+1}_{\bar{\bm{u}}}\|^2\big) \\
	\leq& C_4\tau\sum_{n=0}^{\ell}\big( \|e^{n+1}_{c_1}\|^2 + \|e^{n+1}_{c_2}\|^2 + |e^{n+1}_r|^2 + \|e^n_{\bm{u}}\|^2 + \tau^2\|\nabla e^n_p\|^2 + \|R^{n+1}_{c_1}\|^2_{H^{-1}}  \\
	& \quad + \|R^{n+1}_{c_2}\|^2_{H^{-1}} + \|R^{n+1}_{\bm{u}}\|^2_{H^{-1}} + \|R^{n+1}_p\|^2_{H^{-1}} + \|R^{n+1}_p\|^2 + |R^{n+1}_r|^2\big),
	\end{aligned}
\end{eqnarray} 
for $\ell = 0,1,\ldots,m$ and for some positive constant $C_4$ independent of $\tau$. Hence, there exists a sufficiently small positive constant $\tau_1 \leq \tau_0$ satisfying $C_4\tau_1 \leq 1$ such that when $\tau \leq \tau_1$, we apply the discrete Gronwall lemma \ref{le3.1} to obtain
\begin{eqnarray}\label{hy3.42}
	\begin{aligned}
	& \max_{0 \leq n \leq m}\big( \|e^{n+1}_{c_1}\|^2 + \|e^{n+1}_{c_2}\|^2 + \|e^{n+1}_{\bm{u}}\|^2 + \tau^2\|\nabla e^{n+1}_p\|^2 + |e^{n+1}_r|^2 \big) \\
	& \quad + \sum_{n=0}^{m}\|e^{n+1}_{\bar{\bm{u}}} - e^n_{\bm{u}}\|^2 + \tau\sum_{n=0}^{m}\big(\|\nabla e^{n+1}_{c_1}\|^2 + \|\nabla e^{n+1}_{c_2}\|^2 + \|\nabla e^{n+1}_{\bar{\bm{u}}}\|^2\big)  \\
	\leq& C\tau\sum_{n=0}^{m}\big(\|R^{n+1}_{c_1}\|^2_{H^{-1}} + \|R^{n+1}_{c_2}\|^2_{H^{-1}} + \|R^{n+1}_{\bm{u}}\|^2_{H^{-1}} + \|R^{n+1}_p\|^2_{H^{-1}} + \|R^{n+1}_p\|^2 + |R^{n+1}_r|^2\big) \\
	\leq& C\tau^2.
	\end{aligned}
\end{eqnarray} 
Finally, we derive from \eqref{hy3.24} and \eqref{hy3.42} that
\begin{eqnarray}\label{hy3.43}
	\max_{0 \leq n \leq m}\left(\|\nabla e^{n+1}_{\phi}\|^2 + \|\Delta e^{n+1}_{\phi}\|^2\right) \leq C\max_{0 \leq n \leq m} \left(\|e^{n+1}_{c_1}\|^2 + \|e^{n+1}_{c_2}\|^2\right) \leq C\tau^2.
\end{eqnarray}
Thanks to \eqref{hy3.1}, \eqref{hy3.42} and inequality \cite{temam2001navier}
$$
\|\bm{u}\|_{H^1} = \|\mathbb{P}_H\bar{\bm{u}}\|_{H^1} \leq C(\Omega)\|\bar{\bm{u}}\|_{H^1}, \quad \forall\bm{u},\bar{\bm{u}}\in\bm{H}^1(\Omega),
$$
we have
\begin{eqnarray}\label{hy3.44}
	 \tau\sum_{n=0}^{m}\|e^{n+1}_{\bm{u}}\|^2_{H^1} \leq C\tau\sum_{n=0}^{m}\|e^{n+1}_{\bar{\bm{u}}}\|^2_{H^1} \leq C\tau\sum_{n=0}^{m}\|\nabla e^{n+1}_{\bar{\bm{u}}}\|^2 \leq C\tau^2.
\end{eqnarray}
Combining \eqref{hy3.42} with \eqref{hy3.43} and \eqref{hy3.44}, it leads to the desired result \eqref{hy3.22}.
\end{proof}

We next derive the optimal $H^1$ error estimate for the ionic concentrations $c_1,c_2$.

\begin{lemma}\label{le3.6}
Assume that all the conditions in Theorem \ref{th3.3} hold, then under the induction assumption \eqref{hy3.12}, there exists a sufficiently small positive constant $\tau_2 \leq \tau_1$ such that when $\tau \leq \tau_2$, we have
\begin{eqnarray}\label{hy3.45}
	\max_{0 \leq n \leq m}\left(\|e^{n+1}_{c_1}\|^2_{H^1} + \|e^{n+1}_{c_2}\|^2_{H^1}\right) + \tau\sum_{n=0}^{m}\left(\|e^{n+1}_{c_1}\|^2_{H^2} + \|e^{n+1}_{c_2}\|^2_{H^2}\right) \leq C\tau^2.
\end{eqnarray}
\end{lemma}

\begin{proof}
Rewriting \eqref{hy3.25} and \eqref{hy3.26}, we have
\begin{eqnarray}
	\label{hy3.46} 
	\begin{aligned}
		\delta_t e^{n+1}_{c_1} - \Delta e^{n+1}_{c_1} + \bm{u}^n\cdot\nabla e^{n+1}_{c_1} + e^n_{\bm{u}}\cdot\nabla c_1(t^{n+1}) - \nabla e^{n+1}_1 \cdot \nabla\phi^n \\
	    - e^{n+1}_{c_1}\Delta\phi^n - \nabla c_1(t^{n+1}) \cdot \nabla e^n_{\phi} - c_1(t^{n+1})\Delta e^{n}_{\phi} = R^{n+1}_{c_1}, 
	 \end{aligned}\\
	\label{hy3.47} 
	\begin{aligned}
	\delta_t e^{n+1}_{c_2} - \Delta e^{n+1}_{c_2} + \bm{u}^n\cdot\nabla e^{n+1}_{c_2} + e^n_{\bm{u}}\cdot\nabla c_2(t^{n+1}) + \nabla e^{n+1}_2 \cdot \nabla\phi^n \\
    + e^{n+1}_{c_2}\Delta\phi^n + \nabla c_2(t^{n+1}) \cdot \nabla e^n_{\phi} + c_2(t^{n+1})\Delta e^{n}_{\phi} = R^{n+1}_{c_2},
    \end{aligned}
\end{eqnarray}
Taking the inner product of \eqref{hy3.46} with $-\Delta e^{n+1}_{c_1}$, we obtain
\begin{eqnarray}\label{hy3.48}
	\begin{aligned}
	& \frac{1}{2\tau}\left(\|\nabla e^{n+1}_{c_1}\|^2 - \|\nabla e^n_{c_1}\|^2 + \|\nabla e^{n+1}_{c_1} - \nabla e^n_{c_1}\|^2\right) + \|\Delta e^{n+1}_{c_1}\|^2 \\
	=& \big(\bm{u}^n\cdot\nabla e^{n+1}_{c_1}, \Delta e^{n+1}_{c_1}\big) + \big(e^n_{\bm{u}}\cdot\nabla c_1(t^{n+1}), \Delta e^{n+1}_{c_1}\big) - \big(\nabla e^{n+1}_{c_1}\cdot\nabla\phi^n, \Delta e^{n+1}_{c_1}\big) \\
	&  - \big(e^{n+1}_{c_1} \Delta\phi^n, \Delta e^{n+1}_{c_1}\big) - \big(\nabla c_1(t^{n+1})\cdot\nabla e^n_{\phi}, \Delta e^{n+1}_{c_1}\big) - \big(c_1(t^{n+1}) \Delta e^n_{\phi}, \Delta e^{n+1}_{c_1}\big) \\
	& + \big(R^{n+1}_{c_1}, \Delta e^{n+1}_{c_1}\big) =: \sum_{i=1}^{7}M^{n+1}_i.
	\end{aligned}
\end{eqnarray}
By using \eqref{hy3.3}, \eqref{hy3.12} and \eqref{hy3.14} with $k=2$, we can estimate the terms $M^{n+1}_1$, $M^{n+1}_3$ and $M^{n+1}_4$ as follows
\begin{eqnarray*}
	M^{n+1}_1 &\leq& \|\bm{u}^n\|_{L^4}\|\nabla e^{n+1}_{c_1}\|_{L^4}\|\Delta e^{n+1}_{c_1}\| \leq C\|\bm{u}^n\|^{\frac{1}{2}} \|\nabla\bm{u}^n\|^{\frac{1}{2}} \|\nabla e^{n+1}_{c_1}\|^{\frac{1}{2}} \|\Delta e^{n+1}_{c_1}\|^{\frac{3}{2}} \\
	&\leq& C\epsilon^{-1}\|\nabla e^{n+1}_{c_1}\|^2 + \epsilon\|\Delta e^{n+1}_{c_1}\|^2, \\
	M^{n+1}_3 &\leq& \|\nabla e^{n+1}_{c_1}\|_{L^4} \|\nabla\phi^n\|_{L^4} \|\Delta e^{n+1}_{c_1}\| \leq C\|\nabla e^{n+1}_{c_1}\|^{\frac{1}{2}} \|\phi^n\|_{W^{1,4}} \|\Delta e^{n+1}_{c_1}\|^{\frac{3}{2}} \\
	&\leq& C\epsilon^{-1}\|\nabla e^{n+1}_{c_1}\|^2 + \epsilon\|\Delta e^{n+1}_{c_1}\|^2, \\
	M^{n+1}_4 &\leq& \|e^{n+1}_{c_1}\|_{L^4}\|\Delta\phi^n\|_{L^4}\|\Delta e^{n+1}_{c_1}\| \leq C\epsilon^{-1}\|\nabla e^{n+1}_{c_1}\|^2 + \epsilon\|\Delta e^{n+1}_{c_1}\|^2.
\end{eqnarray*}
By using \eqref{hy3.24}, we have
\begin{eqnarray*}
	M^{n+1}_2 + M^{n+1}_5 + M^{n+1}_6 + M^{n+1}_7 &\leq& \big(\|e^n_{\bm{u}}\| \|\nabla c_1(t^{n+1})\|_{L^{\infty}} + \|\nabla c_1(t^{n+1})\|_{L^{\infty}} \|\nabla e^n_{\phi}\| \\
	&& + \|c_1(t^{n+1})\|_{L^{\infty}} \|\Delta e^n_{\phi}\| + \|R^{n+1}_{c_1}\|\big) \|\Delta e^{n+1}_{c_1}\| \\
	&\leq& C\epsilon^{-1}\big(\|e^n_{\bm{u}}\|^2 + \|e^n_{c_1}\|^2 + \|e^n_{c_2}\|^2 + \|R^{n+1}_{c_2}\|^2\big) + \epsilon\|\Delta e^{n+1}_{c_1}\|^2.
\end{eqnarray*}
Then together with \eqref{hy3.48} and choosing $\epsilon=1/8$, we obtain
\begin{eqnarray}\label{hy3.50}
	\begin{aligned}
	& \frac{1}{2\tau}\left(\|\nabla e^{n+1}_{c_1}\|^2 - \|\nabla e^n_{c_1}\|^2 + \|\nabla e^{n+1}_{c_1} - \nabla e^n_{c_1}\|^2\right) + \frac{1}{2}\|\Delta e^{n+1}_{c_1}\|^2 \\
	\leq& C\big(\|\nabla e^{n+1}_{c_1}\|^2 + \|e^n_{c_1}\|^2 + \|e^n_{c_2}\|^2 + \|e^n_{\bm{u}}\|^2 + \|R^{n+1}_{c_1}\|^2\big).
	\end{aligned}
\end{eqnarray}
Similarly, taking the inner product of \eqref{hy3.47} with $-\Delta e^{n+1}_{c_2}$, we obtain
\begin{eqnarray}\label{hy3.51}
	\begin{aligned}
	&\frac{1}{2\tau}\left(\|\nabla e^{n+1}_{c_2}\|^2 - \|\nabla e^n_{c_2}\|^2 + \|\nabla e^{n+1}_{c_2} - \nabla e^n_{c_2}\|^2\right) + \frac{1}{2}\|\Delta e^{n+1}_{c_2}\|^2 \\
	\leq& C\big(\|\nabla e^{n+1}_{c_2}\|^2 + \|e^n_{c_1}\|^2 + \|e^n_{c_2}\|^2 + \|e^n_{\bm{u}}\|^2 + \|R^{n+1}_{c_2}\|^2\big).
	\end{aligned}
\end{eqnarray}
Then combining \eqref{hy3.50} with \eqref{hy3.51} and summing up the resultant for $n$ from 0 to $\ell$ with $\ell\leq m$, we obtain
\begin{eqnarray}\label{hy3.52}
	\begin{aligned}
	& \max_{0 \leq n \leq \ell} \big(\|\nabla e^{n+1}_{c_1}\|^2 + \|\nabla e^{n+1}_{c_2}\|^2\big) + \tau\sum_{n=0}^{\ell}\big(\|\Delta e^{n+1}_{c_1}\|^2 + \|\Delta e^{n+1}_{c_2}\|^2\big) \\
	\leq& C_5\tau\sum_{n=0}^{\ell}\big( \|\nabla e^{n+1}_{c_1}\|^2 + \|\nabla e^{n+1}_{c_2}\|^2 + \|e^n_{c_1}\|^2 + \|e^n_{c_2}\|^2 + \|e^n_{\bm{u}}\|^2 + \|R^{n+1}_{c_1}\|^2 + \|R^{n+1}_{c_2}\|^2\big),
	\end{aligned}
\end{eqnarray}
for $\ell = 0,1,\ldots,m$ and for some positive constant $C_5$ independent of $\tau$. Hence, there exists a sufficiently small positive constant $\tau_2 \leq \tau_1$ satisfying $C_5\tau_2 \leq 1$ such that when $\tau \leq \tau_2$, we apply the discrete Gronwall lemma \ref{le3.1} and use \eqref{hy3.22} to obtain
\begin{eqnarray}\label{hy3.53}
	\begin{aligned}
	& \max_{0 \leq n \leq m} \big(\|\nabla e^{n+1}_{c_1}\|^2 + \|\nabla e^{n+1}_{c_2}\|^2\big) + \tau\sum_{n=0}^{m}\big(\|\Delta e^{n+1}_{c_1}\|^2 + \|\Delta e^{n+1}_{c_2}\|^2\big) \\
	\leq& C\tau\sum_{n=0}^{m}\big(\|e^n_{c_1}\|^2 + \|e^n_{c_2}\|^2 + \|e^n_{\bm{u}}\|^2 + \|R^{n+1}_{c_1}\|^2 + \|R^{n+1}_{c_2}\|^2\big) \leq C\tau^2,
	\end{aligned}
\end{eqnarray}
which implies the desired result \eqref{hy3.45}.
\end{proof}

\subsection{Proof of Theorem \ref{th3.3}}

To finish the proof of Theorem \ref{th3.3}, we need to prove that \eqref{hy3.12} is valid for $m=n+1$ with $0 \leq n \leq N-1$.

By the error estimates derived in the above Lemmas \ref{le3.4}-\ref{le3.6} in Subsections \ref{subse3.1} and \ref{subse3.2}, we have
\begin{eqnarray*}
	\max_{0 \leq n \leq m}\big(\|e^{n+1}_{c_1}\|_{H^2} + \|e^{n+1}_{c_2}\|_{H^1} + \|e^{n+1}_{\bm{u}}\|_{H^1}\big) \leq C\tau^{\frac{1}{2}}.
\end{eqnarray*}
Then we obtain
\begin{eqnarray*}
	\max_{0 \leq n \leq m}\big\{\|c^{n+1}_1\|_{L^{\infty}}, \|c^{n+1}_2\|_{L^{\infty}}, \|\bm{u}^{n+1}\|_{H^1}\big\} \leq (C_N - 1) + C\tau^{\frac{1}{2}} \leq C_N,
\end{eqnarray*}
for the sufficiently small positive constant $\tau_*$ with $\tau \leq \tau_* \leq \tau_2$ such that $C\tau^{\frac{1}{2}} \leq 1$, which completes the mathematical induction \eqref{hy3.12}. Hence the temporal error estimates in Theorem \ref{th3.3} follow from Lemmas \ref{le3.5}, \ref{le3.6}. This completes the proof of Theorem \ref{th3.3}.

Under the regular condition \eqref{hy3.9}, we derive from \eqref{hy3.10} in Theorem \ref{th3.3} that
\begin{eqnarray}
	\label{hy3.54} \max_{0 \leq n \leq N-1}\left(\|\delta_tc^{n+1}_1\|_{H^1} + \|\delta_tc^{n+1}_2\|_{H^1} + \|\delta_t\phi^{n+1}\|_{H^2} + \|\delta_t\bm{u}^{n+1}\|_{H^1} + \|p^{n+1}\|_{H^1}\right) \leq C,
\end{eqnarray}
which plays a key role in the following error estimate for the pressure $p$.

\subsection{Error estimates for the pressure}

To derive the optimal error estimate of pressure $p$, we first give the following first step error estimates:

\begin{lemma}\label{le3.7}
Under the assumption \eqref{hy3.9} in Theorem \ref{th3.3}, we have
\begin{eqnarray}\label{hy3.55}
	\|\delta_te^1_{c_1}\|^2 + \|\delta_te^1_{c_2}\|^2 + \|\delta_te^1_{\bm{u}}\|^2 + \tau^2\|\nabla\delta_te^1_p\|^2 \leq C\tau^2.
\end{eqnarray}
\end{lemma}

\begin{proof}
By taking $n=0$ in \eqref{hy3.25} and using $e^0_{c_1} = e^0_{\phi} = e^0_{\bm{u}} = 0$, we have
\begin{eqnarray}\label{hy3.56}
	\frac{1}{\tau}e^1_{c_1} - \Delta e^1_{c_1} + \bm{u}^0\cdot\nabla e^1_{c_1} - \nabla\cdot(e^1_{c_1}\nabla\phi(t^0)) = R^1_{c_1}.
\end{eqnarray}
Taking the inner product of \eqref{hy3.56} with $\frac{1}{\tau}e^1_{c_1}$, we obtain
\begin{eqnarray*}
	\|\frac{e^1_{c_1}}{\tau}\|^2 + \frac{1}{\tau}\|\nabla e^1_{c_1}\|^2 &=& -\frac{1}{\tau}\big(e^1_{c_1}\nabla\phi(t^0), \nabla e^1_{c_1}\big) + \frac{1}{\tau}(R^1_{c_1}, e^1_{c_1}) \\
	&\leq& C\big(\|\nabla e^1_{c_1}\|^2 + \|R^1_{c_1}\|^2\big) + \frac{1}{2}\|\frac{e^1_{c_1}}{\tau}\|^2,
\end{eqnarray*}
which leads that
\begin{eqnarray}\label{hy3.57}
	\|\delta_t e^1_{c_1}\|^2 + \frac{2}{\tau}\|\nabla e^1_{c_1}\|^2 \leq C\tau^2.
\end{eqnarray}
Similarly to \eqref{hy3.57}, we take $n=0$ in \eqref{hy3.26} and take the inner product of the resultant with $\frac{1}{\tau}e^1_{c_2}$ to obtain
\begin{eqnarray}\label{hy3.58}
	\|\delta_t e^1_{c_2}\|^2 + \frac{2}{\tau}\|\nabla e^1_{c_2}\|^2 \leq C\tau^2.
\end{eqnarray}

By taking $n=0$ in \eqref{hy3.30} and using $e^0_p = e^0_{\bm{u}} = 0$, we have
\begin{eqnarray}\label{hy3.59}
	\frac{1}{\tau}e^1_{\bar{\bm{u}}} - \Delta e^1_{\bar{\bm{u}}} + \frac{e^1_r}{\sqrt{E(\phi^1)}}\big(\bm{u}^0\cdot\nabla\bm{u}^0 + (c^0_1 - c^0_2)\nabla\phi^0\big) = J^1_1 + J^1_4.
\end{eqnarray}
Taking the inner product of \eqref{hy3.59} with $\frac{1}{\tau}e^1_{\bar{\bm{u}}}$, we have
\begin{eqnarray}\label{hy3.60}
	\begin{aligned}
	\|\frac{e^1_{\bar{\bm{u}}}}{\tau}\|^2 + \frac{1}{\tau}\|\nabla e^1_{\bar{\bm{u}}}\|^2 =& -\frac{e^1_r}{\tau\sqrt{E(\phi^1)}}\big(\bm{u}^0\cdot\nabla\bm{u}^0 + (c^0_1 - c^0_2)\nabla\phi^0, e^1_{\bar{\bm{u}}}\big) + \frac{1}{\tau}\big(J^1_1 + J^1_4, e^1_{\bar{\bm{u}}}\big) \\
	\leq& C\big(\|e^1_{c_1}\|^2 + \|e^1_{c_2}\|^2 + |e^1_r|^2 + \|R^1_{\bm{u}}\|^2\big) + \frac{1}{8}\|\frac{e^1_{\bar{\bm{u}}}}{\tau}\|^2.
	\end{aligned}
\end{eqnarray}
By taking $n=0$ in \eqref{hy3.33} and using $e^0_p=0$, we have
\begin{eqnarray}\label{hy3.61}
	\frac{1}{\tau}e^1_{\bm{u}} + \nabla e^1_p = \frac{1}{\tau} e^1_{\bar{\bm{u}}} + R^1_p.
\end{eqnarray}
Taking the inner product of \eqref{hy3.61} with itself on both sides, we have
\begin{eqnarray}\label{hy3.62}
	\frac{1}{2}\Big(\|\frac{e^1_{\bm{u}}}{\tau}\|^2 - \|\frac{e^1_{\bar{\bm{u}}}}{\tau}\|^2\Big) + \frac{1}{2}\|\nabla e^1_p\|^2 = \frac{1}{2}\|R^1_p\|^2 + \frac{1}{\tau}(e^1_{\bar{\bm{u}}}, R^1_p)  \leq C\|R^1_p\|^2 + \frac{1}{8}\|\frac{e^1_{\bar{\bm{u}}}}{\tau}\|^2. 
\end{eqnarray}
Together \eqref{hy3.60} with \eqref{hy3.62} and using $e^0_{\bm{u}} = e^0_p = 0$, we obtain 
\begin{eqnarray}\label{hy3.63}
	\|\delta_t e^1_{\bm{u}}\|^2 + \frac{1}{2}\|\frac{e^1_{\bar{\bm{u}}}}{\tau}\|^2 + \frac{2}{\tau}\|\nabla e^1_{\bar{\bm{u}}}\|^2 + \tau^2\|\nabla\delta_te^1_p\|^2 \leq C\tau^2.
\end{eqnarray}
Combining \eqref{hy3.57}, \eqref{hy3.58} and \eqref{hy3.63}, it leads to the desired result \eqref{hy3.55}.
\end{proof}

\begin{lemma}\label{le3.8}
Under the assumption \eqref{hy3.9} in Theorem \ref{th3.3}, there exists a sufficiently small constant $\tau_{**} \leq \tau_*$ such that when $\tau \leq \tau_{**}$, we have 
\begin{eqnarray}\label{hy3.64}
	\begin{aligned}
	\max_{0 \leq n \leq N-1}\big( \|\delta_te^{n+1}_{c_1}\|^2 + \|\delta_te^{n+1}_{c_2}\|^2 + \|\delta_te^{n+1}_{\bm{u}}\|^2 
	+ \tau^2\|\delta_te^{n+1}_p\|^2_{H^1}\big) \\
	+ \tau\sum_{n=0}^{N-1}\big( \|\delta_te^{n+1}_{c_1}\|^2_{H^1} + \|\delta_te^{n+1}_{c_2}\|^2_{H^1} + \|\delta_t e^{n+1}_{\bar{\bm{u}}}\|^2_{H^1}\big) \leq C\tau^2.
	\end{aligned}
\end{eqnarray}
\end{lemma}

\begin{proof}
We shall follow four steps.

{\bf Step 1.} Taking the difference of two consecutive steps in \eqref{hy3.23}, we have
\begin{eqnarray}\label{hy3.65}
	-\Delta\delta_te^{n+1}_{\phi} = \delta_te^{n+1}_{c_1} - \delta_te^{n+1}_{c_2}.
\end{eqnarray}
Then taking the inner product of \eqref{hy3.65} with $\delta_te^{n+1}_{\phi}$ and $-\Delta\delta_te^{n+1}_{\phi}$, we have
\begin{eqnarray*}
	&& \|\nabla\delta_te^{n+1}_{\phi}\|^2 \leq \big(\|\delta_te^{n+1}_{c_1}\| + \|\delta_te^{n+1}_{c_2}\|\big)\|\delta_te^{n+1}_{\phi}\|, \\ 
	&& \|\Delta\delta_te^{n+1}_{\phi}\|^2 \leq \big(\|\delta_te^{n+1}_{c_1}\| + \|\delta_te^{n+1}_{c_2}\|\big)\|\Delta\delta_te^{n+1}_{\phi}\|,
\end{eqnarray*}
which imply that
\begin{eqnarray}\label{hy3.66}
	\|\nabla\delta_te^{n+1}_{\phi}\|^2 + \|\Delta\delta_te^{n+1}_{\phi}\|^2 \leq C\big(\|\delta_te^{n+1}_{c_1}\|^2 + \|\delta_te^{n+1}_{c_2}\|^2\big).
\end{eqnarray}

{\bf Step 2.} Taking the difference of two consecutive steps in \eqref{hy3.25} and in \eqref{hy3.26} respectively, we have
\begin{eqnarray}
	\label{hy3.67} 
	\begin{aligned}
	&\delta^2_te^{n+1}_{c_1} - \Delta\delta_te^{n+1}_{c_1} + \bm{u}^n\cdot\nabla\delta_te^{n+1}_{c_1} + \delta_t\bm{u}^n\cdot\nabla e^n_{c_1} + e^n_{\bm{u}}\cdot\nabla\delta_tc_1(t^{n+1}) + \delta_te^n_{\bm{u}}\cdot\nabla c_1(t^n) \\
	&  - \nabla\cdot(\delta_tc^{n+1}_1\nabla e^n_{\phi}) - \nabla\cdot(c^n_1\nabla\delta_te^n_{\phi}) - \nabla\cdot(\delta_te^{n+1}_{c_1}\nabla\phi(t^n)) - \nabla\cdot(e^n_{c_1}\nabla\delta_t\phi(t^n)) = \delta_tR^{n+1}_{c_1}, 
	\end{aligned}\\
	\label{hy3.68} 
	\begin{aligned}
	& \delta^2_te^{n+1}_{c_2} - \Delta\delta_te^{n+1}_{c_2} + \bm{u}^n\cdot\nabla\delta_te^{n+1}_{c_2} + \delta_t\bm{u}^n\cdot\nabla e^n_{c_2} + e^n_{\bm{u}}\cdot\nabla\delta_tc_2(t^{n+1}) + \delta_te^n_{\bm{u}}\cdot\nabla c_2(t^n) \\
	&  - \nabla\cdot(\delta_tc^{n+1}_2\nabla e^n_{\phi}) - \nabla\cdot(c^n_2\nabla\delta_te^n_{\phi}) - \nabla\cdot(\delta_te^{n+1}_{c_2}\nabla\phi(t^n)) - \nabla\cdot(e^n_{c_2}\nabla\delta_t\phi(t^n)) = \delta_tR^{n+1}_{c_2}.
	\end{aligned}
\end{eqnarray}
Taking the inner product of \eqref{hy3.67} with $\delta_te^{n+1}_{c_1}$, we have
\begin{eqnarray}\label{hy3.69}
	\begin{aligned}
	& \frac{1}{2\tau}\big(\|\delta_te^{n+1}_{c_1}\|^2 - \|\delta_te^{n}_{c_1}\|^2 + \|\delta_te^{n+1}_{c_1} - \delta_te^n_{c_1}\|^2\big) + \|\nabla\delta_te^{n+1}_{c_1}\|^2 \\
	=& - \big(\delta_t\bm{u}^n\cdot\nabla e^n_{c_1}, \delta_te^{n+1}_{c_1}\big) - \big(e^n_{\bm{u}}\cdot\nabla\delta_tc_1(t^{n+1})\big) - \big(\delta_te^n_{\bm{u}}\cdot\nabla c_1(t^n), \delta_te^{n+1}_{c_1}\big)  \\
	& - \big(\delta_tc^{n+1}_1\nabla e^n_{\phi}, \delta_te^{n+1}_{c_1}\big) - \big(c^n\nabla\delta_te^n_{\phi}, \delta_te^{n+1}_{c_1}\big) - \big(\delta_te^{n+1}_{c_1}\nabla\phi(t^n), \delta_te^{n+1}_{c_1}\big) \\
	& - \big(e^n_{c_1}\nabla\delta_t\phi(t^n), \delta_te^{n+1}_{c_1}\big) + \big(R^{n+1}_{c_1}, \delta_te^{n+1}_{c_1}\big) =: \sum_{i=1}^8Q^{n+1}_i.
	\end{aligned}
\end{eqnarray}
By using \eqref{hy3.3}, \eqref{hy3.24}, \eqref{hy3.54} and \eqref{hy3.66}, we can estimate the terms $Q^{n+1}_i$ as follows
\begin{eqnarray*}
	Q^{n+1}_1 &\leq& \|\delta_t\bm{u}^n\|_{L^4} \|\nabla e^n_{c_1}\| \|\delta_te^{n+1}_{c_1}\|_{L^4} \leq C\epsilon^{-1}\|\nabla e^n_{c_1}\|^2 + \epsilon\|\nabla\delta_te^{n+1}_{c_1}\|^2,\\
	Q^{n+1}_2 + Q^{n+1}_3 &\leq& \big(\|e^n_{\bm{u}}\|\|\nabla\delta_tc_1(t^{n+1})\|_{L^{\infty}} + \|\delta_te^n_{\bm{u}}\|\|\nabla c_1(t^n)\|_{L^{\infty}}\big) \|\delta_te^{n+1}_{c_1}\| \\
	&\leq& C\epsilon^{-1}\big(\|\delta_te^n_{\bm{u}}\|^2 + \|e^n_{\bm{u}}\|^2\big) + \epsilon\|\nabla\delta_t e^{n+1}_{c_1}\|^2,\\
	Q^{n+1}_4 + Q^{n+1}_5 &\leq& \big(\|\delta_tc^{n+1}_1\|_{L^4}\|\nabla e^n_{\phi}\| + \|c^n_1\|_{L^4}\|\nabla\delta_te^n_{\phi}\|\big)\|\delta_te^{n+1}_{c_1}\|_{L^4} \\
	&\leq& C\epsilon^{-1}\big(\|\delta_te^n_{c_1}\|^2 + \|\delta_te^n_{c_2}\|^2 + \|e^n_{c_1}\|^2 + \|e^n_{c_2}\|^2\big) + \epsilon\|\nabla\delta_te^{n+1}_{c_1}\|^2,\\
	Q^{n+1}_6 + Q^{n+1}_7 + Q^{n+1}_8 &\leq& \big(\|\delta_te^{n+1}_{c_1}\|\|\nabla\phi(t^n)\|_{L^{\infty}} + \|e^n_{c_1}\|\|\nabla\delta_t\phi(t^n)\|_{L^{\infty}} + \|\delta_tR^{n+1}_{c_1}\|\big) \|\delta_te^{n+1}_{c_1}\| \\
	&\leq& C\epsilon^{-1}\big(\|\delta_te^{n+1}_{c_1}\|^2 + \|e^n_{c_1}\|^2 + \|\delta_tR^{n+1}_{c_1}\|^2\big) + \epsilon\|\nabla\delta_te^{n+1}_{c_1}\|^2.
\end{eqnarray*}
Then together with \eqref{hy3.69} and choosing $\epsilon=1/8$, we have
\begin{eqnarray}\label{hy3.70}
	\begin{aligned}
	\frac{1}{2\tau}\big(\|\delta_te^{n+1}_{c_1}\|^2 - \|\delta_te^{n}_{c_1}\|^2\big) + \frac{1}{2}\|\nabla\delta_te^{n+1}_{c_1}\|^2 
	\leq C\big( \|\delta_te^{n+1}_{c_1}\|^2 + \|\delta_te^n_{c_1}\|^2 + \|\delta_te^n_{c_2}\|^2  \\ 
	 + \|\delta_te^n_{\bm{u}}\|^2 + \|e^n_{c_1}\|^2 + \|e^n_{c_2}\|^2 + \|\nabla e^n_{c_1}\|^2 + \|\delta_tR^{n+1}_{c_1}\|^2\big).
	\end{aligned}
\end{eqnarray}
Similarly, taking the inner product of \eqref{hy3.68} with $\delta_te^{n+1}_{c_2}$, we have
\begin{eqnarray}\label{hy3.71}
	\begin{aligned}
	\frac{1}{2\tau}\big(\|\delta_te^{n+1}_{c_2}\|^2 - \|\delta_te^{n}_{c_2}\|^2 \big) + \frac{1}{2}\|\nabla\delta_te^{n+1}_{c_2}\|^2 
	\leq C\big( \|\delta_te^{n+1}_{c_2}\|^2 + \|\delta_te^n_{c_1}\|^2 + \|\delta_te^n_{c_2}\|^2  \\
	+ \|\delta_te^n_{\bm{u}}\|^2 + \|e^n_{c_1}\|^2 + \|e^n_{c_2}\|^2 + \|\nabla e^n_{c_2}\|^2 + \|\delta_tR^{n+1}_{c_2}\|^2\big).
	\end{aligned}
\end{eqnarray}

{\bf Step 3.} Taking the difference of two consecutive steps in \eqref{hy3.30}, we have
\begin{eqnarray}\label{hy3.72}
	\begin{aligned}
	& \frac{\delta_te^{n+1}_{\bar{\bm{u}}} - \delta_te^n_{\bm{u}}}{\tau} - \Delta\delta_te^{n+1}_{\bar{\bm{u}}} + \nabla\delta_te^n_p \\
	=& - \delta_t\Big(\frac{e^{n+1}_r}{\sqrt{E(\phi^{n+1})}}\big(\bm{u}^n\cdot\nabla\bm{u}^n + (c^n_1-c^n_2)\nabla\phi^n\Big) + \sum_{i=1}^4\delta_tJ^{n+1}_i.
	\end{aligned}
\end{eqnarray}
Then taking the inner product of \eqref{hy3.72} with $\delta_te^{n+1}_{\bar{\bm{u}}}$, we have
\begin{eqnarray}\label{hy3.73}
	\begin{aligned}
	& \frac{1}{2\tau}\big(\|\delta_te^{n+1}_{\bar{\bm{u}}}\|^2 - \|\delta_te^n_{\bm{u}}\|^2 + \|\delta_te^{n+1}_{\bar{\bm{u}}} - \delta_te^n_{\bm{u}}\|^2\big) + \|\nabla\delta_te^{n+1}_{\bar{\bm{u}}}\|^2 + \big(\nabla\delta_te^n_p, \delta_te^{n+1}_{\bar{\bm{u}}}\big) \\
	=& - \left( \delta_t\Big(\frac{e^{n+1}_r}{\sqrt{E(\phi^{n+1})}}\big(\bm{u}^n\cdot\nabla\bm{u}^n + (c^n_1-c^n_2)\nabla\phi^n\big)\Big), \delta_te^{n+1}_{\bar{\bm{u}}} \right) + \sum_{i=1}^4\big(\delta_tJ^{n+1}_i, \delta_te^{n+1}_{\bar{\bm{u}}}\big).
	\end{aligned}
\end{eqnarray}
We estimate the first term of the right hand side of \eqref{hy3.73} from \eqref{hy3.12}, \eqref{hy3.14} and \eqref{hy3.54} that
\begin{eqnarray*}
	&& \left( \delta_t\Big(\frac{e^{n+1}_r}{\sqrt{E(\phi^{n+1})}}\big(\bm{u}^n\cdot\nabla\bm{u}^n + (c^n_1-c^n_2)\nabla\phi^n\big)\Big), \delta_te^{n+1}_{\bar{\bm{u}}} \right) \\
	&=& \Big(\frac{\delta_te^{n+1}_r}{\sqrt{E(\phi^{n+1})}} + \frac{\delta_tE(\phi^{n+1}) e^{n+1}_r}{\sqrt{E(\phi^{n+1})E(\phi^{n})\big(E(\phi^{n+1}) + E(\phi^{n})\big)}} \Big) \big( \bm{u}^n\cdot\nabla\bm{u}^n + (c^n_1-c^n_2)\nabla\phi^n, \delta_te^{n+1}_{\bar{\bm{u}}}\big) \\
	&& + \frac{e^n_r}{\sqrt{E(\phi^n)}} \big( \delta_t\bm{u}^n\cdot\bm{u}^n + \bm{u}^{n-1}\cdot\nabla\delta_t\bm{u}^n + (\delta_tc^n_1 - \delta_tc^n_2)\nabla\phi^n + (c^{n-1}_1 - c^{n-1}_2)\nabla\delta_t\phi^n, \delta_te^{n+1}_{\bar{\bm{u}}}\big) \\
	&\leq& C\epsilon^{-1}\big(|\delta_te^{n+1}_r|^2 + |e^{n+1}_r|^2 + |e^n_r|^2\big) + \epsilon\|\nabla\delta_t e^{n+1}_{\bar{\bm{u}}}\|^2,
\end{eqnarray*}
and 
\begin{eqnarray*}
	&& \big(\delta_tJ^{n+1}_1, \delta_te^{n+1}_{\bar{\bm{u}}}\big) \\
	&=& \Big(\delta_t\Big(\frac{1}{\sqrt{E(\phi^{n+1})}} - \frac{1}{\sqrt{E(\phi(t^{n+1}))}}\Big) r(t^{n+1})\big(\bm{u}(t^n)\cdot\nabla\bm{u}(t^n) \\
	&& + (c_1(t^n) - c_2(t^n))\nabla\phi(t^n)\big), \delta_te^{n+1}_{\bar{\bm{u}}} \Big) + \Big( \Big(\frac{1}{\sqrt{E(\phi^{n})}} - \frac{1}{\sqrt{E(\phi(t^{n}))}}\Big) \\
	&& \times \delta_t\big(r(t^{n+1})\big(\bm{u}(t^n)\cdot\nabla\bm{u}(t^n) + (c_1(t^n) - c_2(t^n))\nabla\phi(t^n)\big)\big), \delta_te^{n+1}_{\bar{\bm{u}}} \Big) \\
	&\leq& C\big(\|\nabla\delta_te^{n+1}_{\phi}\| + \|\nabla e^{n+1}_{\phi}\| + \|\nabla e^{n}_{\phi}\|\big)\|\delta_te^{n+1}_{\bar{\bm{u}}}\| \\
	&\leq& C\epsilon^{-1}\big(\|\delta_te^{n+1}_{c_1}\|^2 + \|\delta_te^{n+1}_{c_2}\|^2 + \|e^{n+1}_{c_1}\|^2 + \|e^{n+1}_{c_2}\|^2 + \|e^{n}_{c_1}\|^2 + \|e^{n}_{c_2}\|^2\big) + \epsilon\|\nabla\delta_t e^{n+1}_{\bar{\bm{u}}}\|^2,
\end{eqnarray*}
where we use 
\begin{eqnarray*}
	&& \delta_tE(\phi^{n+1}) \leq \|\nabla\delta_t\phi^{n+1}\| (\|\nabla\phi^{n+1}\| + \|\nabla\phi^n\|) \leq C, \\
	&& \delta_t(E(\phi(t^{n+1})) - E(\phi^{n+1})) \leq \big(\|\nabla e^{n+1}_{\phi}\| + \|\nabla e^n_{\phi}\|\big) \|\nabla\delta_t\phi(t^{n+1})\| + \|\nabla\delta_te^{n+1}_{\phi}\| (\|\nabla\phi^{n+1}\| + \|\nabla\phi^n\|) \\
	&& \leq C\big(\|\nabla\delta_te^{n+1}_{\phi}\| + \|\nabla e^{n+1}_{\phi}\| + \|\nabla e^n_{\phi}\| \big).
\end{eqnarray*}
By using \eqref{hy3.5} and \eqref{hy3.44}, we can estimate the third term as follows
\begin{eqnarray*}
	&& \big(\delta_tJ^{n+1}_2, \delta_te^{n+1}_{\bar{\bm{u}}}\big) \\
	&=& - \Big( \Big(\frac{\delta_tr(t^{n+1})}{\sqrt{E(\phi^{n+1})}} + \frac{\delta_tE(\phi^{n+1}) r(t^n)}{\sqrt{E(\phi^{n+1})E(\phi^{n})\big(E(\phi^{n+1}) + E(\phi^{n})\big)}}\Big) \\
	&& \times \big(\bm{u}(t^n)\cdot\nabla\bm{u}(t^n) - \bm{u}^n\cdot\nabla\bm{u}^n \big), \delta_te^{n+1}_{\bar{\bm{u}}}\Big) - \frac{r(t^n)}{\sqrt{E(\phi^n)}}\big( \delta_t\bm{u}(t^n)\cdot\nabla\bm{u}(t^n) - \delta_t\bm{u}^n\cdot\nabla\bm{u}^n \\
	&& + \bm{u}(t^{n-1})\cdot\nabla\delta_t\bm{u}(t^n) - \bm{u}^{n-1}\cdot\nabla\delta_t\bm{u}^n, \delta_te^{n+1}_{\bar{\bm{u}}}\big) \\
	&\leq& C\epsilon^{-1}\big(\|\delta_te^n_{\bm{u}}\|^2 + \|e^n_{\bm{u}}\|^2\big) + \epsilon\|\nabla\delta_te^{n+1}_{\bar{\bm{u}}}\|^2,
\end{eqnarray*}
where we use $\|\nabla\delta_te^n_{\bm{u}}\| \leq C\tau^{\frac{1}{2}} \leq C$. Similarly, we have
\begin{eqnarray*}
	\big(\delta_tJ^{n+1}_3, \delta_te^{n+1}_{\bar{\bm{u}}}\big) &\leq& C\epsilon^{-1}\big(\|\delta_te^n_{c_1}\|^2 + \|\delta_te^n_{c_2}\|^2\big) + \epsilon\|\nabla\delta_te^{n+1}_{\bar{\bm{u}}}\|^2, \\
	\big(\delta_tJ^{n+1}_4, \delta_te^{n+1}_{\bar{\bm{u}}}\big) &\leq& C\epsilon^{-1}\|\delta_tR^{n+1}_{\bm{u}}\|^2_{H^{-1}} + \epsilon\|\nabla\delta_te^{n+1}_{\bar{\bm{u}}}\|^2.
\end{eqnarray*}

From \eqref{hy3.35}, we have
\begin{eqnarray}\label{hy3.74}
	\delta_te^{n+1}_{\bm{u}} + \tau\nabla\delta_t e^{n+1}_p = \delta_te^{n+1}_{\bar{\bm{u}}} + \tau\nabla\delta_t e^n_p + \tau\delta_t R^{n+1}_p.
\end{eqnarray}
Taking the inner product of \eqref{hy3.74} with itself on both sides, we have
\begin{eqnarray}\label{hy3.75}
	\begin{aligned}
	& \|\delta_te^{n+1}_{\bm{u}} \|^2 + \tau^2\|\nabla \delta_te^{n+1}_p\|^2 \\
	=& \|\delta_te^{n+1}_{\bar{\bm{u}}}\|^2 + \tau^2\|\nabla \delta_te^n_p\|^2 + \tau^2\|\delta_tR^{n+1}_p\|^2 + 2\tau(\nabla \delta_te^n_p,\delta_te^{n+1}_{\bar{\bm{u}}})\\
	& + 2\tau(\delta_te^{n+1}_{\bar{\bm{u}}}, \delta_tR^{n+1}_p) + 2\tau^2(\nabla \delta_te^n_p, \delta_tR^{n+1}_p). 
	\end{aligned}
\end{eqnarray}
The terms $\tau(\nabla \delta_te^n_p, \delta_tR^{n+1}_p)$ and $(\delta_te^{n+1}_{\bar{\bm{u}}}, \delta_tR^{n+1}_p)$ can be estimated as follows
\begin{eqnarray*}
	&& \tau(\nabla \delta_te^n_p, \delta_tR^{n+1}_p) \leq C\big(\tau^2\|\nabla\delta_t e^n_p\|^2 + \|\delta_tR^{n+1}_p\|^2\big), \\
	&& (\delta_te^{n+1}_{\bar{\bm{u}}},\delta_tR^{n+1}_p) \leq C\epsilon^{-1}\|\delta_tR^{n+1}_p\|^2_{H^{-1}} + \epsilon\|\nabla \delta_te^{n+1}_{\bar{\bm{u}}}\|^2.
\end{eqnarray*}

Then together \eqref{hy3.73} with \eqref{hy3.75} and choosing $\epsilon = 1/12$, we obtain
\begin{eqnarray}\label{hy3.76}
	\begin{aligned}
	& \frac{1}{2\tau}\big(\|\delta_te^{n+1}_{\bm{u}}\|^2 - \|\delta_te^n_{\bm{u}}\|^2 \big) + \frac{\tau}{2}\big(\|\nabla\delta_te^{n+1}_p\|^2 - \|\nabla\delta_te^n_p\|^2\big) + \frac{1}{2}\|\nabla\delta_te^{n+1}_{\bar{\bm{u}}}\|^2  \\
	\leq& C\big(\|\delta_te^{n+1}_{c_1}\|^2 + \|\delta_te^{n+1}_{c_2}\|^2 + \|\delta_te^{n}_{c_1}\|^2 + \|\delta_te^{n}_{c_2}\|^2 + \|e^{n+1}_{c_1}\|^2 + \|e^{n+1}_{c_2}\|^2 \\
	& + \|e^{n}_{c_1}\|^2 + \|e^{n}_{c_2}\|^2 + \|\delta_te^n_{\bm{u}}\|^2 + \|e^n_{\bm{u}}\|^2 + \tau^2\|\nabla\delta_t e^n_p\|^2 + |\delta_te^{n+1}_r|^2  \\
	&  + |e^{n+1}_r|^2 + |e^n_r|^2 + \|\delta_tR^{n+1}_p\|^2 + \|\delta_tR^{n+1}_{\bm{u}}\|^2\big).
	\end{aligned}
\end{eqnarray}

{\bf Step 4.} From \eqref{hy3.38}, we have
\begin{eqnarray}\label{hy3.77}
	\begin{aligned}
	|\delta_te^{n+1}_r| \leq& \frac{1}{2\sqrt{E(\phi^{n+1})}} \big|\big(\bm{u}^n\cdot\nabla\bm{u}^n + (c^n_1-c^n_2)\nabla\phi^n, e^{n+1}_{\bar{\bm{u}}}\big)\big| + \sum_{i=1}^6|K^{n+1}_i| + |R^{n+1}_r| \\
	\leq& C\big(\|e^{n+1}_{c_1}\| + \|e^{n+1}_{c_2}\| + \|e^{n}_{c_1}\| + \|e^{n}_{c_2}\| + \|\nabla e^{n+1}_{\bar{\bm{u}}}\| + \|e^n_{\bm{u}}\| + |e^{n+1}_r| + |R^{n+1}_r|\big).
	\end{aligned}
\end{eqnarray}
Then together with \eqref{hy3.70}, \eqref{hy3.71}, \eqref{hy3.76} and \eqref{hy3.77}, we have
\begin{eqnarray}\label{hy3.78}
	\begin{aligned}
	& \|\delta_te^{n+1}_{c_1}\|^2 - \|\delta_te^{n}_{c_1}\|^2 + \|\delta_te^{n+1}_{c_2}\|^2 - \|\delta_te^{n}_{c_2}\|^2 + \|\delta_te^{n+1}_{\bm{u}}\|^2 - \|\delta_te^n_{\bm{u}}\|^2\\
	& + \tau^2\big(\|\nabla\delta_te^{n+1}_p\|^2 - \|\nabla\delta_te^n_p\|^2\big) + \tau\|\nabla\delta_te^{n+1}_{c_1}\|^2 + \tau\|\nabla\delta_te^{n+1}_{c_2}\|^2 + \tau\|\nabla\delta_te^{n+1}_{\bar{\bm{u}}}\|^2 \\
	\leq& C\tau\big(\|\delta_te^{n+1}_{c_1}\|^2 + \|\delta_te^{n+1}_{c_2}\|^2 + \|\delta_te^{n}_{c_1}\|^2 + \|\delta_te^{n}_{c_2}\|^2 + \|e^{n+1}_{c_1}\|^2 + \|e^{n+1}_{c_2}\|^2 + \|e^n_{c_1}\|^2    \\
	& + \|e^n_{c_2}\|^2 + \|\nabla e^n_{c_1}\|^2 + \|\nabla e^n_{c_2}\|^2 + \|\delta_te^n_{\bm{u}}\|^2 + \|e^n_{\bm{u}}\|^2 + \|\nabla e^{n+1}_{\bar{\bm{u}}}\|^2 + \tau^2\|\nabla\delta_t e^n_p\|^2  \\
	& + |e^{n+1}_r|^2 + |e^n_r|^2 + \|\delta_tR^{n+1}_{c_1}\|^2 + \|\delta_tR^{n+1}_{c_2}\|^2 + \|\delta_tR^{n+1}_p\|^2 + \|\delta_tR^{n+1}_{\bm{u}}\|^2 + |R^{n+1}_r|^2 \big).
	\end{aligned}
\end{eqnarray}
Summing up \eqref{hy3.78} for $n$ from 0 to $\ell$ with $\ell \leq N-1$, we use \eqref{hy3.10} and \eqref{hy3.11} to obtain 
\begin{eqnarray}\label{hy3.79}
	\begin{aligned}
	& \max_{0 \leq n \leq \ell} \big(\|\delta_te^{n+1}_{c_1}\|^2 + \|\delta_te^{n}_{c_2}\|^2 + \|\delta_te^{n+1}_{\bm{u}}\|^2 + \tau^2\|\nabla\delta_te^{n+1}_p\|^2 \big)  \\
	&\qquad + \tau\sum_{n=0}^{\ell}\big(\|\nabla\delta_te^{n+1}_{c_1}\|^2 + \|\nabla\delta_te^{n+1}_{c_2}\|^2 + \tau\|\nabla\delta_te^{n+1}_{\bar{\bm{u}}}\|^2\big)  \\
	\leq& \|\delta_te^1_{c_1}\|^2 + \|\delta_te^1_{c_2}\|^2 + \|\delta_te^1_{\bm{u}}\|^2 + \tau^2\|\nabla\delta_te^1_p\|^2  \\
	& + C_6\tau\sum_{n=0}^{\ell}\big( \|\delta_te^{n+1}_{c_1}\|^2 + \|\delta_te^{n+1}_{c_2}\|^2 + \|e^{n+1}_{c_1}\|^2 + \|e^{n+1}_{c_2}\|^2 + \|\nabla e^n_{c_1}\|^2 + \|\nabla e^n_{c_2}\|^2  \\
	& + \|\delta_te^n_{\bm{u}}\|^2 + \|e^n_{\bm{u}}\|^2 + \|\nabla e^{n+1}_{\bar{\bm{u}}}\|^2 + \tau^2\|\nabla\delta_t e^n_p\|^2 + |e^{n+1}_r|^2 + \|\delta_tR^{n+1}_{c_1}\|^2  \\
	&  + \|\delta_tR^{n+1}_{c_2}\|^2 + \|\delta_tR^{n+1}_p\|^2 + \|\delta_tR^{n+1}_{\bm{u}}\|^2 + |R^{n+1}_r|^2 \big)  \\
	\leq& C_6\tau\sum_{n=0}^{\ell}\big( \|\delta_te^{n+1}_{c_1}\|^2 + \|\delta_te^{n+1}_{c_2}\|^2 + \|\delta_te^n_{\bm{u}}\|^2 + \tau^2\|\nabla\delta_t e^n_p\|^2 \big) + C\tau^2,
	\end{aligned}
\end{eqnarray}
for $\ell = 0,1,\ldots,N-1$ and for some positive constant $C_6$ independent of $\tau$. Hence, there exists a sufficiently small positive constant $\tau_{**} \leq \tau_*$ satisfying $C_6\tau_{**} \leq 1$ such that when $\tau \leq \tau_{**}$, we apply the discrete Gronwall lemma \ref{le3.1} to obtain
\begin{eqnarray}\label{hy3.80}
	\begin{aligned}
	 \max_{0 \leq n \leq N-1}\big( \|\delta_te^{n+1}_{c_1}\|^2 + \|\delta_te^{n+1}_{c_2}\|^2 + \|\delta_te^{n+1}_{\bm{u}}\|^2 
	+ \tau^2\|\nabla\delta_te^{n+1}_p\|^2\big) \\
	+ \tau\sum_{n=0}^{N-1}\big( \|\nabla\delta_te^{n+1}_{c_1}\|^2 + \|\nabla\delta_te^{n+1}_{c_2}\|^2 + \|\nabla\delta_t e^{n+1}_{\bar{\bm{u}}}\|^2\big) \leq C\tau^2,
	\end{aligned}
\end{eqnarray}
which implies the desired result \eqref{hy3.64}.
\end{proof}

Now, we can prove the optimal error estimate for pressure. Meantime, the $H^1$ error estimate for velocity is also obtained.

\begin{theorem}\label{th3.9}
	Under the assumption \eqref{hy3.9} in Theorem \ref{th3.3} and the conditions in Lemma \ref{le3.8}, we have
	\begin{eqnarray}\label{hy3.81}
		\max_{0 \leq n \leq N-1}\big(\|e^{n+1}_{\bar{\bm{u}}}\|^2_{H^1} + \|e^{n+1}_{\bm{u}}\|^2_{H^1} + \|e^{n+1}_p\|^2\big) \leq C\tau^2.
	\end{eqnarray}
\end{theorem}

\begin{proof}
From \eqref{hy3.64} in Lemma \ref{le3.8} and using the H\"older's inequality, for $0 \leq n \leq \ell$ with $\ell \leq N-1$, we have
\begin{eqnarray*}
	\|\nabla e^{\ell+1}_{\bar{\bm{u}}}\| - \|\nabla e^{1}_{\bar{\bm{u}}}\| = \sum_{n=1}^{\ell}\big(\|\nabla e^{n+1}_{\bar{\bm{u}}}\| - \|\nabla e^{n}_{\bar{\bm{u}}}\|\big) \leq \Big(\tau^2\sum_{n=0}^{\ell} \|\nabla \delta_t e^{n+1}_{\bar{\bm{u}}}\|^2\Big)^{\frac{1}{2}} \Big(\sum_{n=0}^{\ell}1^2\Big)^{\frac{1}{2}} \leq C\tau,
\end{eqnarray*}
for $\ell = 0,1,\ldots,N-1$. Then it follows from 
\begin{eqnarray}\label{hy3.82}
	\max_{0 \leq n \leq N-1}\|e^{n+1}_{\bar{\bm{u}}}\|_{H^1} \leq C\tau + C\max_{0 \leq n \leq N-1}\|\nabla e^{1}_{\bar{\bm{u}}}\| \leq C\tau.
\end{eqnarray}
Thanks to $e^{n+1}_{\bm{u}} = \mathbb{P}_He^{n+1}_{\bar{\bm{u}}}$, we have
\begin{eqnarray}\label{hy3.83}
	\max_{0 \leq n \leq N-1}\|e^{n+1}_{\bm{u}}\|_{H^1} \leq C\max_{0 \leq n \leq N-1}\|\mathbb{P}_H e^{n+1}_{\bar{\bm{u}}}\|_{H^1} \leq C\max_{0 \leq n \leq N-1} \|e^{n+1}_{\bar{\bm{u}}}\|_{H^1} \leq C\tau.
\end{eqnarray}

Adding \eqref{hy3.30} and \eqref{hy3.33}, it leads to 
\begin{eqnarray}\label{hy3.84}
	\delta_te^{n+1}_{\bm{u}} - \Delta e^{n+1}_{\bar{\bm{u}}} + \nabla e^{n+1}_p + \frac{e^{n+1}_r}{\sqrt{E(\phi^{n+1})}} \big(\bm{u}^n\cdot\nabla\bm{u}^n + (c^n_1-c^n_2)\nabla\phi^n\big) = \sum_{i=1}^4J^{n+1}_i + R^{n+1}_p.
\end{eqnarray}
By using the inf-sup condition
$$
C_7\|p\| \leq \sup\limits_{\bm{u} \in \bm{H}^1_0(\Omega)} \frac{(\nabla\cdot\bm{u}, p)}{\|\nabla\bm{u}\|}, \quad \forall p\in L^2_0(\Omega),
$$
for some positive constant $C_7$, we obtain from \eqref{hy3.84} that
\begin{eqnarray}\label{hy3.85}
	\begin{aligned}
	\max_{0 \leq n \leq N-1}\|e^{n+1}_p\| \leq& C\max_{0 \leq n \leq N-1} \big(\|\delta_te^{n+1}_{\bm{u}}\| + \|\nabla e^{n+1}_{\bar{\bm{u}}}\| + |e^{n+1}_r| + \sum_{i=1}^4\|J^{n+1}_i\| + \|R^{n+1}_p\|\big)  \\
	\leq& C\max_{0 \leq n \leq N-1} \big(\|e^{n+1}_{c_1}\| + \|e^{n+1}_{c_2}\| + \|e^{n}_{c_1}\| + \|e^{n}_{c_2}\| + \|e^n_{\bm{u}}\| + \|\delta_te^{n+1}_{\bm{u}}\|  \\
	& + \|\nabla e^{n+1}_{\bar{\bm{u}}}\| + |e^{n+1}_r| + \|R^{n+1}_{\bm{u}}\| + \|R^{n+1}_p\|\big) \\
	\leq& C\tau.
	\end{aligned}
\end{eqnarray}
Combining \eqref{hy3.82}, \eqref{hy3.83} and \eqref{hy3.85}, we obtain the desired result \eqref{hy3.81}.
\end{proof}

\section{Numerical experiments}\label{se4}

In this section, we provide some numerical experiments to validate the proposed scheme \eqref{hy2.8}. The spatial discretizations are based on the finite element method which is very efficient and accurate. The fully discrete finite element scheme for \eqref{hy2.6} reads as follows:

{\bf Step 1.} Find $\phi^0_h \in V_h$ such that for all $\psi_h\in V_h$
\begin{eqnarray}
	\big(\nabla\phi^0_h, \nabla\psi_h\big) = \big(c^0_{1,h} - c^0_{2,h}, \psi_h\big).
\end{eqnarray}

{\bf Step 2.} Find $c^{n+1}_{1,h}, c^{n+1}_{2,h} \in V_h \subset H^1(\Omega)$ such that for all $\theta_h \in V_h$
\begin{eqnarray}
	\label{hy4.1} && \big(\delta_tc^{n+1}_{1,h}, \theta_h\big) + \big(\nabla c^{n+1}_{1,h}, \nabla\theta_h\big) - \big(\bm{u}^n_h c^{n+1}_{1,h}, \nabla\theta_h\big) + \big(c^{n+1}_{1,h}\nabla\phi^n_h, \nabla\theta_h\big) = 0, \\
	\label{hy4.2} && \big(\delta_tc^{n+1}_{2,h}, \theta_h\big) + \big(\nabla c^{n+1}_{2,h}, \nabla\theta_h\big) - \big(\bm{u}^n_h c^{n+1}_{2,h}, \nabla\theta_h\big) - \big(c^{n+1}_{2,h}\nabla\phi^n_h, \nabla\theta_h\big) = 0.
\end{eqnarray}

{\bf Step 3.} Find $\phi^{n+1}_h \in V_h$ such that for all $\psi_h\in V_h$
\begin{eqnarray}\label{hy4.3}
	\big(\nabla\phi^{n+1}_h, \nabla\psi_h\big) = \big(c^{n+1}_{1,h} - c^{n+1}_{2,h}, \psi_h\big).
\end{eqnarray}

{\bf Step 4.} Find $\bar{\bm{u}}^{n+1}_h \in \bm{X}_h \subset \bm{H}^1_0(\Omega)$ such that for all $\bm{v}_h \in \bm{X}_h$
\begin{eqnarray}\label{hy4.4} 
	\begin{aligned}
	\Big(\frac{\bar{\bm{u}}^{n+1}_h-\bm{u}^n_h}{\tau}, \bm{v}_h\Big) +  \big(\nabla\bar{\bm{u}}^{n+1}_h, \nabla\bm{v}_h\big) - \big(p^n_h, \nabla\cdot\bm{v}_h\big) + \frac{r^{n+1}_h}{\sqrt{E(\phi^{n+1}_h)}} \big(\bm{u}^n_h\cdot\nabla\bm{u}^n_h, \bm{v}_h\big)  \\
	+ \frac{r^{n+1}_h}{\sqrt{E(\phi^{n+1}_h)}}\big((c^{n}_{1,h} - c^{n}_{2,h})\nabla\phi^{n}_h, \bm{v}_h\big) = 0,
	\end{aligned}
\end{eqnarray}
and find $r^{n+1} \in \mathbb{R}$ such that
\begin{eqnarray}\label{hy4.5}
	\begin{aligned}
	\delta_tr^{n+1}_h =& \frac{1}{2\sqrt{E(\phi^{n+1}_h)}} \Big((\bm{u}^n_h\cdot\nabla\bm{u}^n_h,\bar{\bm{u}}^{n+1}_h) + (c^{n}_{1,h}-c^{n}_{2,h})\nabla\phi^{n}_h,\bar{\bm{u}}^{n+1}_h)\Big) \\
	& - \frac{1}{2r^{n+1}_h}\Big(\|c^{n+1}_{1,h}-c^{n+1}_{2,h}\|^2 + \int_{\Omega}\big(c^{n+1}_{1,h}+c^{n+1}_{2,h}\big)|\nabla\phi^{n+1}_h|^2d\bm{x}\Big).
	\end{aligned}
\end{eqnarray}

{\bf Step 5.} Find $p^{n+1}_h \in Y_h \subset L^2_0(\Omega)$ such that for all $q_h \in Y_h$
\begin{eqnarray}\label{hy4.6}
	\big(\nabla p^{n+1}_h, \nabla q_h\big) = - \frac{1}{\tau} \big(\nabla\cdot\bar{\bm{u}}^{n+1}_h, q_h\big) + \big(\nabla p^n_h, \nabla q_h\big),
\end{eqnarray}
and update $\bm{u}^{n+1}_h$ from 
\begin{eqnarray}\label{hy4.7}
	\bm{u}^{n+1}_h = \bar{\bm{u}}^{n+1}_h - \tau\nabla p^{n+1}_h + \tau\nabla p^n_h.
\end{eqnarray}

Similar to Theorem \ref{th2.1}, we can directly prove the above fully discrete scheme is also mass preserving and unconditionally energy stable in the senses that
\begin{eqnarray}\label{hy4.8}
	\int_{\Omega}c^{n+1}_{1,h} d\bm{x} = \int_{\Omega}c^{n}_{1,h} d\bm{x}, \quad \int_{\Omega}c^{n+1}_{2,h} d\bm{x} = \int_{\Omega}c^{n}_{2,h} d\bm{x},
\end{eqnarray}
and 
\begin{eqnarray}\label{hy4.9}
	E^{n+1}_h - E^n_h \leq -\tau\|\nabla\bar{\bm{u}}^{n+1}_h\|^2 - \tau\|c^{n+1}_{1,h}-c^{n+1}_{2,h}\|^2 - \tau\int_{\Omega}\big(c^{n+1}_{1,h}+c^{n+1}_{2,h}\big)|\nabla\phi^{n+1}_h|^2d\bm{x},
\end{eqnarray}
where 
$$
E^{n+1}_h = \frac{1}{2}\|\bm{u}^{n+1}_h\|^2 + \frac{\tau^2}{2}\|\nabla p^{n+1}_h\|^2 + |r^{n+1}_h|^2.
$$

In the following numerical experiments, we adopt $\mathcal{P}_1$ element for the ionic concentrations $c_1$, $c_2$ and the electric potential $\phi$ and Taylor-Hood element $(\bm{\mathcal{P}}_2, \mathcal{P}_1)$ for $(\bm{u}, p)$.

{\bf Example 1.} We first consider the following quite artificial exact solutions on the computed domain $\Omega = (-1,1)^2$:
\begin{eqnarray*}
	&& c_1 = 3\cos(\pi x)\cos(\pi y)\sin(t), \quad c_2 = \cos(\pi x)\cos(\pi y)\sin(t), \\
	&& \phi = \frac{1}{\pi^2}\cos(\pi x)\cos(\pi y)\sin(t), \quad p = \sin(2\pi x)\sin(2\pi y)\sin(t), \\
	&& \bm{u} = \big(\sin(2\pi x)\cos(2\pi y)\sin(t), -\sin(2\pi y)\cos(2\pi x)\sin(t)\big)^{\top},
\end{eqnarray*}
where the ionic concentrations $c_1$ and $c_2$ are even negative. The suitable source terms are added in the NSPNP equations \eqref{hy1.1} for the above exact solutions. We set $C_0 = 10$ and fix the spatial size $h=0.05$. The numerical results of errors and convergence rates at $t=1$ are listed in Tables \ref{tab1}-\ref{tab2}. We observe that all the error values achieve the first-order temporal convergence rates.

\begin{table}
	\centering
	\caption{Errors and convergence rates in the $L^2$-norm at $t=1$ for Example 1.}\label{tab1}
	\begin{tabular}{c c c c c c c c c c c}
		\hline
		$\tau$ & $\|e_{c_1}\|$ & Rate & $\|e_{c_2}\|$ & Rate & $\|e_{\phi}\|$ & Rate & $\|e_{\bm{u}}\|$ & Rate & $\|e_p\|$ & Rate \\ \hline 
		$\frac{1}{10}$     & 1.76E-1 & $-$  & 5.89E-2 & $-$  & 6.30E-3 & $-$  & 8.51E-2 & $-$  & 5.28E-2 & $-$ \\
		$\frac{1}{20}$    & 8.49E-2 & 1.05 & 2.85E-2 & 1.03 & 3.16E-3 & 1.00 & 4.10E-2 & 1.05 & 2.04E-2 & 1.34 \\
		$\frac{1}{40}$   & 4.17E-2 & 1.03 & 1.40E-2 & 1.03 & 1.65E-3 & 0.94 & 1.98E-2 & 1.05 & 9.04E-3 & 1.17 \\
		$\frac{1}{80}$  & 2.06E-2 & 1.01 & 6.94E-3 & 1.01 & 9.13E-4 & 0.85 & 9.74E-3 & 1.02 & 4.11E-3 & 1.14 \\
		\hline
	\end{tabular}
\end{table}

\begin{table}
	\centering
	\caption{Errors and convergence rates in the $H^1$-norm at $t=1$ for Example 1.}\label{tab2}
	\begin{tabular}{c c c c c c c c c}
		\hline
		$\tau$ & $\|e_{c_1}\|_{H^1}$ & Rate & $\|e_{c_2}\|_{H^1}$ & Rate & $\|e_{\phi}\|_{H^1}$ & Rate & $\|e_{\bm{u}}\|_{H^1}$ & Rate \\ \hline 
		$\frac{1}{10}$     & 8.09E-2 & $-$  & 2.71E-2 & $-$  & 2.85E-3 & $-$  & 9.23E-2 & $-$  \\
		$\frac{1}{20}$    & 4.00E-2 & 1.02 & 1.34E-2 & 1.02 & 1.45E-3 & 0.98 & 4.09E-2 & 1.18 \\
		$\frac{1}{40}$   & 2.13E-2 & 0.91 & 7.13E-3 & 0.91 & 7.94E-4 & 0.87 & 1.81E-2 & 1.17 \\
		$\frac{1}{80}$  & 9.32E-3 & 1.19 & 3.43E-3 & 1.05 & 4.05E-4 & 0.97 & 8.81E-3 & 1.04 \\
		\hline
	\end{tabular}
\end{table}

{\bf Example 2.} We consider the following quite artificial exact solutions on the computed domain $\Omega = (-1,1)^2$:
\begin{eqnarray*}
	&& c_1 = 1.1 + \cos(\pi x)\cos(\pi y)\sin^2(t), \quad c_2 = 1.1 - \cos(\pi x)\cos(\pi y)\sin^2(t), \\
	&& \phi = \frac{1}{\pi^2}\cos(\pi x)\cos(\pi y)\sin^2(t), \quad p = \sin(2\pi x)\sin(2\pi y)\sin^2(t), \\
	&& \bm{u} = \big(\pi\sin(2\pi x)\cos(2\pi y)\sin^2(t), -\pi\sin(2\pi y)\cos(2\pi x)\sin^2(t)\big)^{\top},
\end{eqnarray*}
where the ionic concentrations $c_1$ and $c_2$ are positive. The suitable source terms are added in the NSPNP equations \eqref{hy1.1} for the above exact solutions.  We set $C_0 = 10$ and fix the spatial size $h=0.025$. The numerical results of errors and convergence rates at $t=0.1$ are listed in Tables \ref{tab3}-\ref{tab4}. We observe that the results are consistent with the error estimates in Theorems \ref{th3.3} and \ref{th3.9}.

\begin{table}
	\centering
	\caption{Errors and convergence rates in the $L^2$-norm at $t=0.1$ for Example 2.}\label{tab3}
	\begin{tabular}{c c c c c c c c c c c}
		\hline
		$\tau$ & $\|e_{c_1}\|$ & Rate & $\|e_{c_2}\|$ & Rate & $\|e_{\phi}\|$ & Rate & $\|e_{\bm{u}}\|$ & Rate & $\|e_p\|$ & Rate \\ \hline 
		$\frac{1}{100}$     & 1.67E-3 & $-$  & 1.67E-3 & $-$  & 1.79E-4 & $-$  & 6.52E-3 & $-$  & 2.70E-3 & $-$ \\
		$\frac{1}{200}$    & 8.60E-4 & 0.95 & 8.60E-3 & 0.95 & 9.23E-5 & 0.95 & 3.35E-3 & 0.96 & 1.26E-3 & 1.10 \\
		$\frac{1}{400}$   & 4.42E-4 & 0.96 & 4.42E-4 & 0.96 & 4.75E-5 & 0.96 & 1.70E-3 & 0.98 & 5.75E-4 & 1.13 \\
		$\frac{1}{800}$  & 2.29E-4 & 0.95 & 2.29E-4 & 0.95 & 2.48E-5 & 0.94 & 8.56E-4 & 0.99 & 2.62E-4 & 1.14 \\
		\hline
	\end{tabular}
\end{table}

\begin{table}
	\centering
	\caption{Errors and convergence rates in the $H^1$-norm at $t=0.1$ for Example 2.}\label{tab4}
	\begin{tabular}{c c c c c c c c c}
		\hline
		$\tau$ & $\|e_{c_1}\|_{H^1}$ & Rate & $\|e_{c_2}\|_{H^1}$ & Rate & $\|e_{\phi}\|_{H^1}$ & Rate & $\|e_{\bm{u}}\|_{H^1}$ & Rate \\ \hline 
		$\frac{1}{100}$     & 7.65E-3 & $-$  & 7.65E-3 & $-$  & 7.90E-4 & $-$  & 4.81E-2 & $-$  \\
		$\frac{1}{200}$    & 3.96E-3 & 0.95 & 3.96E-3 & 0.95 & 4.11E-4 & 0.94 & 2.47E-2 & 0.96 \\
		$\frac{1}{400}$   & 2.06E-3 & 0.95 & 2.06E-3 & 0.95 & 2.16E-4 & 0.93 & 1.25E-2 & 0.98 \\
		$\frac{1}{800}$  & 1.11E-3 & 0.89 & 1.11E-3 & 0.89 & 1.18E-4 & 0.87 & 6.31E-3 & 0.99 \\
		\hline
	\end{tabular}
\end{table}

{\bf Example 3.} We verify the properties of mass conservation and non-negativity of the ionic concentrations and energy stability. We consider the following initial conditions on the computed domain $\Omega = (0,1)^2$:
\begin{eqnarray*}
	&& c_1(\bm{x},0) = \cos(\pi x) + 1, \quad c_2(\bm{x},0) = \cos(\pi y) + 1, \\
	&& \bm{u}(\bm{x},0) = (\pi\sin(\pi x)\cos(\pi y), -\pi\sin(\pi y)\cos(\pi x))^{\top},
\end{eqnarray*}
where the initial ionic concentrations are set non-negative. We set $C_0 = 5$ and fix the spatial size $h=0.01$. Figure \ref{fig_energy} plots the time evolutions of the discrete energy \eqref{hy4.9} and original energy \eqref{hy1.6} with different time-steps $\tau=0.1$, 0.05, 0.01 and 0.005. We observe that all energy curves decay monotonically for all time-steps, which confirms that the proposed scheme is unconditionally energy stable and the original energy is dissipative. Furthermore, Figure \ref{fig_mass} plots the time evolutions of masses $\int_{\Omega}c_1 d\bm{x}$, $\int_{\Omega}c_2d\bm{x}$ and their errors with $\tau=0.005$. Figure \ref{fig_c1c2} plots the time evolutions of of maximum and minimum values of $c_1$, $c_2$ with $\tau=0.005$.  It is observed that the proposed scheme preserves the properties of mass conservation and positivity during the time evolution.

\begin{figure}
	\centering
	\subfloat[Discrete energy \eqref{hy4.9}]{
		\includegraphics[scale=0.5]{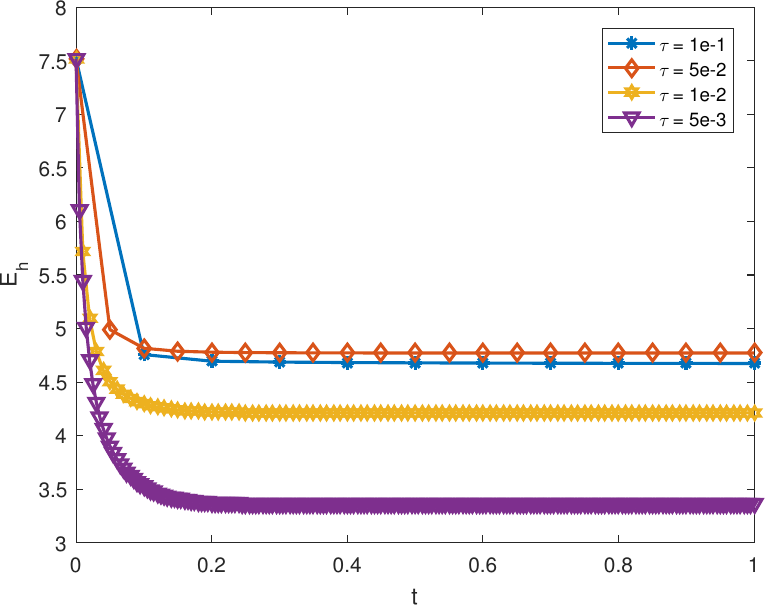} 
	}
	\subfloat[Original energy \eqref{hy1.6}]{
		\includegraphics[scale=0.5]{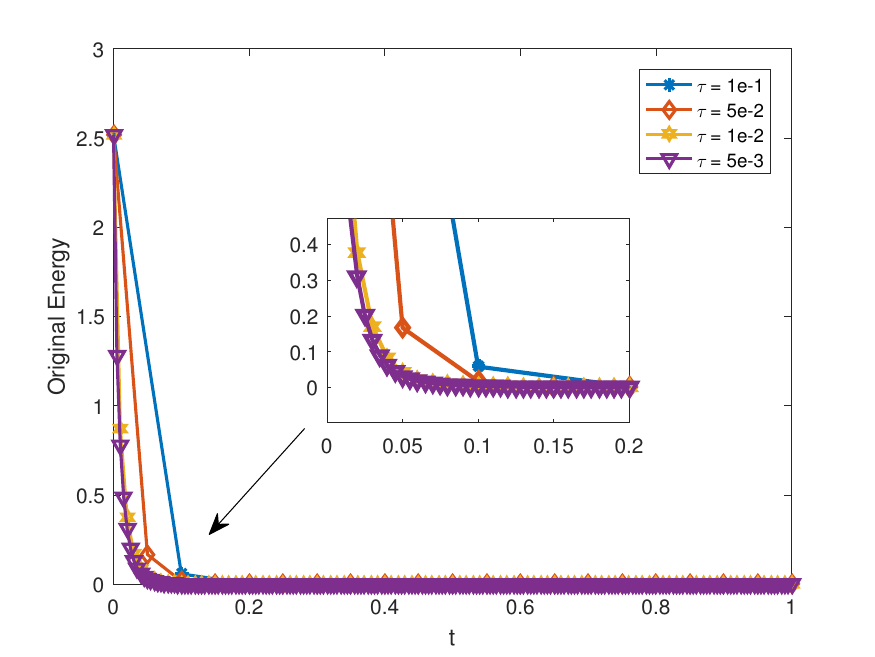}
	}\\
	\caption{Time evolutions of the discrete energy and original energy for Example 3.} \label{fig_energy}
\end{figure}

\begin{figure}
	\centering
		\includegraphics[scale=0.5]{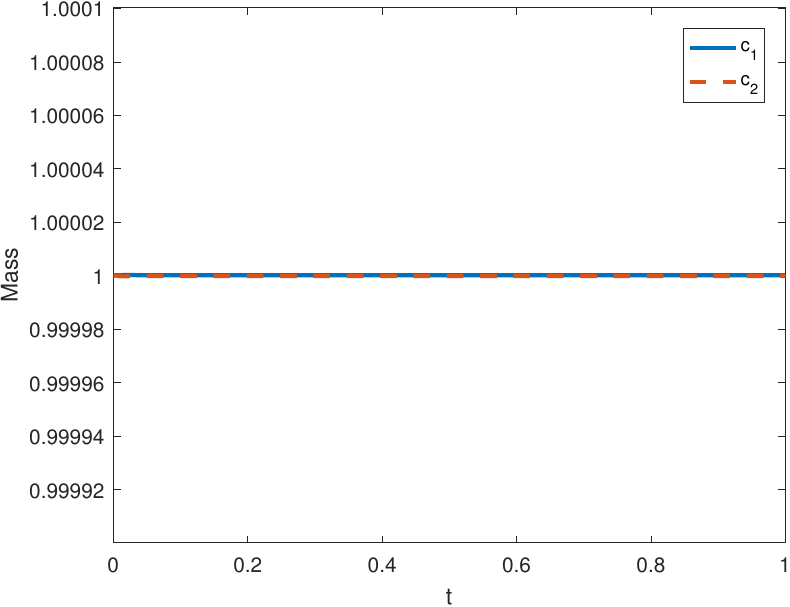} 
		\includegraphics[scale=0.5]{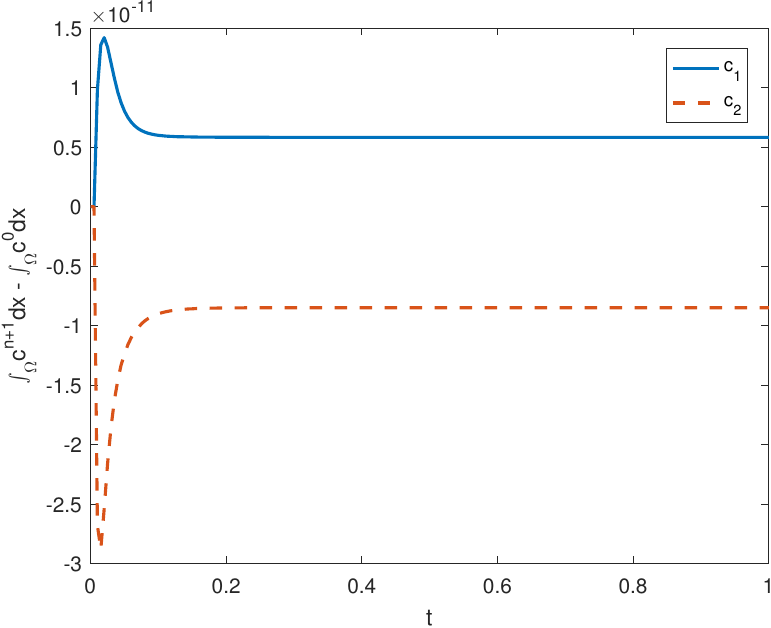}\\
	\caption{Time evolutions of masses $\int_{\Omega}c_1d\bm{x}$, $\int_{\Omega}c_2d\bm{x}$ and their errors with $\tau=0.005$ for Example 3.} \label{fig_mass}
\end{figure}

\begin{figure}
	\centering
	\includegraphics[scale=0.43]{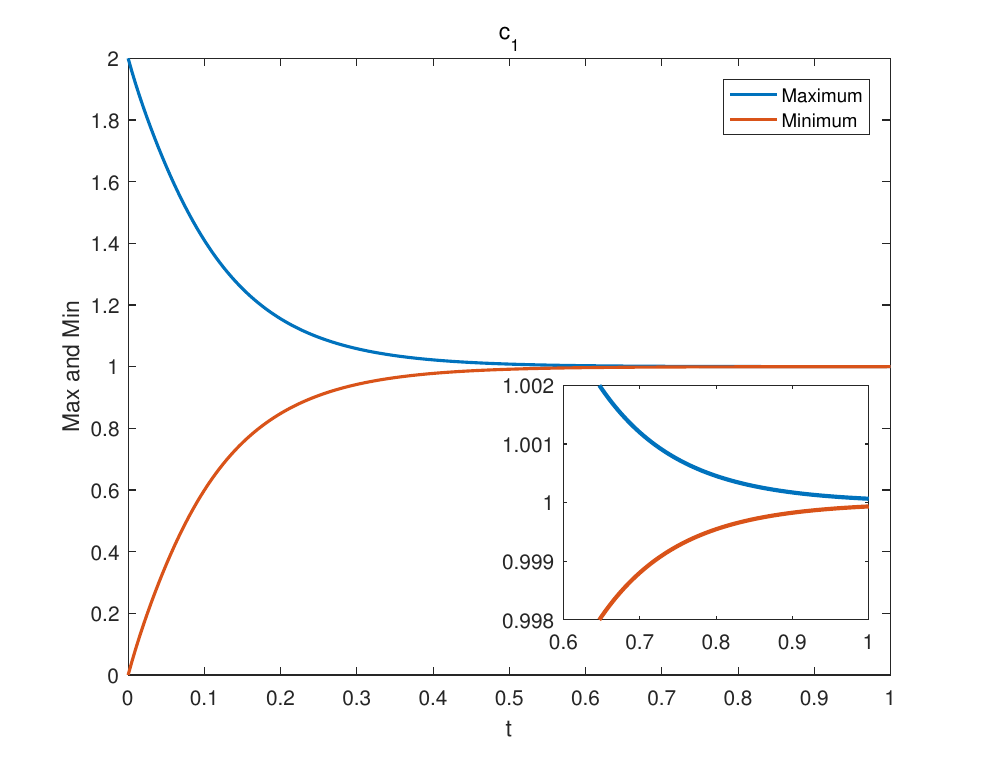} 
	\includegraphics[scale=0.43]{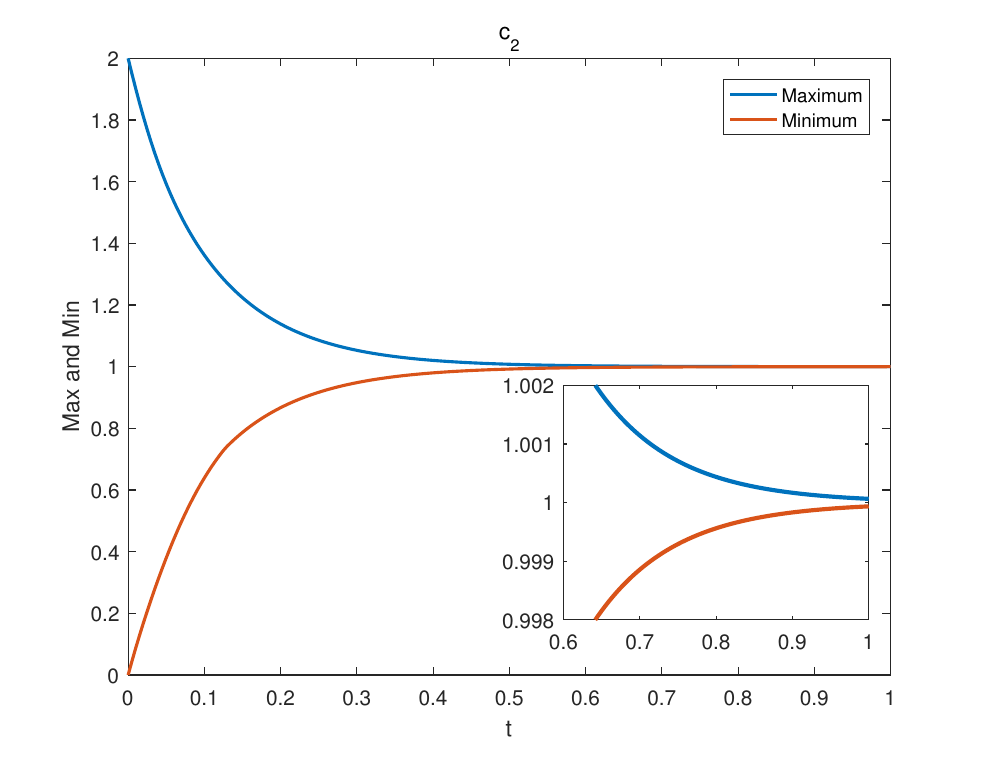}\\
	\caption{Time evolutions of maximum and minimum values of $c_1$, $c_2$ with $\tau=0.005$ for Example 3.} \label{fig_c1c2}
\end{figure}

\section{Concluding remarks}\label{se5}

In this paper, we constructed first-order fully-decoupled and linearized pressure-correction projection scheme for the NSPNP equations and derived that the proposed scheme is unconditionally energy stable, and preserves non-negativity and mass conservation of the ionic concentration solutions. 

We introduced a suitable SAV $r(t)$, which derived from the ionic concentration equations, to treat explicitly the nonlinear terms in the Navier-Stokes part. This allows the nonlinear contributions to the energy to cancel each other out in the continuous case, overcoming the stability condition $\tau \leq \|c^{n+1}_1 + c^{n+1}_2\|^{-1}_{L^{\infty}}$ caused by the temporal discretization \cite{liu2017efficient} of the ionic concentration equations and leading to unconditional energy stability.

We only carry out the rigorous error analysis for the proposed first-order semi-discrete scheme in the two-dimensional case by using mathematical induction on the $L^{\infty}$ bounds for $c_1,c_2$ and $H^1$ bound for $\bm{u}$, which are not available through energy stability. Our analysis essentially uses some inequalities that are valid only in two dimensional case, leading to the fact that they cannot be easily extended to the three dimensional case. We shall leave the error estimates in the three dimensional case for future work.

\bibliographystyle{siamplain}
\bibliography{bib_NSPNP}
\end{document}